 \documentclass[12pt]{amsart}
 \usepackage{amsmath,amsfonts,mathrsfs,amsthm,amssymb}
 \usepackage{amssymb}
 \usepackage{graphicx}
 \usepackage{lscape}
 \usepackage[all,cmtip]{xy}
 \usepackage{cancel}
 \usepackage{bm}
 \usepackage{multirow}
 \usepackage{framed}
 \usepackage{color}
 \usepackage{hyperref}
 \usepackage{youngtab}  \Yvcentermath1
 \usepackage{underscore}
 
 \usepackage{tikz}
 \usepackage{ifthen}
 \usepackage{xspace}
 \input{xstring}
\newcommand\myscale{0.8}

\newcommand\tcirc[3]{
	\ifthenelse{\equal{#1}{w}}{\filldraw[fill=white,draw=black] (#2) circle (0.075);}{}%
	\ifthenelse{\equal{#1}{b}}{\filldraw[black] (#2) circle (0.075);}{}%
	\draw (#2) ++(0,0.35) node {$#3$};
	}

\newcommand\tdots[1]{\draw (#1) ++(0.55,0) node {$\cdots$}}

\newcommand\bond[1]{\draw (#1) -- +(1,0)}

\newcommand\vbond[1]{\draw (#1) -- +(0,-1)}

\newcommand\diagbond[2]{
	\ifthenelse{\equal{#1}{u}}{
		\draw (#2) -- +(0.5,0.865);
	}{}
	\ifthenelse{\equal{#1}{d}}{
		\draw (#2) -- +(0.5,-0.865);
	}{}
	}

\newcommand\dbond[2]{
	\draw (#2) ++(0.03,0.03) -- +(0.94,0);
	\draw (#2) ++(0.03,-0.03) -- +(0.94,0);
	\ifthenelse{\equal{#1}{r}}{
		\draw[semithick] (#2) ++(0.6,0) ++(-0.15,0.2) -- ++(0.15,-0.2) -- +(-0.15,-0.2);
	}{}
	\ifthenelse{\equal{#1}{l}}{
		\draw[semithick] (#2) ++(0.45,0) ++(0.15,0.2) -- ++(-0.15,-0.2) -- +(0.15,-0.2);
	}{}
	}

\newcommand\tbond[2]{
	\draw (#2)  -- +(1,0);
	\draw (#2) ++(0.05,0.06) -- +(0.9,0);
	\draw (#2) ++(0.05,-0.06) -- +(0.9,0);
	\ifthenelse{\equal{#1}{r}}{
		\draw[semithick] (#2) ++(0.6,0) ++(-0.15,0.2) -- ++(0.15,-0.2) -- +(-0.15,-0.2);
	}{}
	\ifthenelse{\equal{#1}{l}}{
		\draw[semithick] (#2) ++(0.45,0) ++(0.15,0.2) -- ++(-0.15,-0.2) -- +(0.15,-0.2);
	}{}
	}

\newcommand\tcross[2]{
	\draw (#1) ++(0,0.35) node {$#2$};
	\draw[semithick] (#1) ++(-0.15,-0.15)-- +(0.3,0.3);
	\draw[semithick] (#1) ++(-0.15,0.15)-- +(0.3,-0.3);
	}

\newcommand\tsquare[2]{
		\draw[semithick,color=blue] (#1) ++(-0.15,-0.15) rectangle ++(0.3,0.3);
		\tcross{#1}{#2};
		}

\newcommand\tstar[2]{
	\draw[color=red] (#1) node {\Large$*$};
	\draw (#1) ++(0,0.35) node {$#2$};
	}

\newcommand\DDnode[3]{
\ifthenelse{\equal{#1}{w}}{\tcirc{w}{#2}{#3}}{}		
\ifthenelse{\equal{#1}{b}}{\tcirc{b}{#2}{#3}}{}		
\ifthenelse{\equal{#1}{x}}{\tcross{#2}{#3}}{}		
\ifthenelse{\equal{#1}{s}}{\tstar{#2}{#3}}{}		
\ifthenelse{\equal{#1}{q}}{\tsquare{#2}{#3}}{}		
}

\newcommand\Edd[2]{
 \begin{tiny}
 \begin{tikzpicture}[scale=\myscale,baseline=-3pt]
 \foreach \x in {0,1,2,3} {
	\bond{\x,0};
 }
 \vbond{2,0};
 
 \StrLen{#1}[\Ernk]
 
 \StrChar{#1}{1}[\nodetype];
 \DDnode{\nodetype}{0,0}{\StrBefore{#2}{,}};
 \StrChar{#1}{2}[\nodetype];
 \DDnode{\nodetype}{2,-1}{\StrBetween[1,2]{#2}{,}{,}};
 \StrChar{#1}{3}[\nodetype];
 \DDnode{\nodetype}{1,0}{\StrBetween[2,3]{#2}{,}{,}};
 \StrChar{#1}{4}[\nodetype];
 \DDnode{\nodetype}{2,0}{\StrBetween[3,4]{#2}{,}{,}};
 \StrChar{#1}{5}[\nodetype];
 \DDnode{\nodetype}{3,0}{\StrBetween[4,5]{#2}{,}{,}};
 \StrChar{#1}{6}[\nodetype];

 \ifthenelse{\equal{\Ernk}{6}}{
 		\DDnode{\nodetype}{4,0}{\StrBehind[5]{#2}{,}};
 		\useasboundingbox (-.4,-1.2) rectangle (4.4,0.55);
	}{}%
 
 \ifthenelse{\equal{\Ernk}{7}}{
 		\bond{4,0};
 		\DDnode{\nodetype}{4,0}{\StrBetween[5,6]{#2}{,}{,}};
		\StrChar{#1}{7}[\nodetype];
		\DDnode{\nodetype}{5,0}{\StrBehind[6]{#2}{,}};
 		\useasboundingbox (-.4,-1.2) rectangle (5.4,0.55);
	}{}%

 \ifthenelse{\equal{\Ernk}{8}}{
 		\bond{4,0};
 		\bond{5,0};
 		\DDnode{\nodetype}{4,0}{\StrBetween[5,6]{#2}{,}{,}};
		\StrChar{#1}{7}[\nodetype];
		\DDnode{\nodetype}{5,0}{\StrBetween[6,7]{#2}{,}{,}};
		\StrChar{#1}{8}[\nodetype];
		\DDnode{\nodetype}{6,0}{\StrBehind[7]{#2}{,}};
		\useasboundingbox (-.4,-1.2) rectangle (6.4,0.55);
	}{}%

 \end{tikzpicture}
 \end{tiny}
 }

 \newcommand\tad{\mathrm{ad}}

 \renewcommand\l{\left}
 \renewcommand\r{\right}

 \newcommand\bop{\bigoplus}
 \newcommand\op{\oplus}
 \newcommand\ot{\otimes}

 \newcommand\fann{\mathfrak{ann}}

 \newcommand\rnk{\mathrm{rank}}

 \newcommand\fa{{\mathfrak a}}

 \newcommand\ff{{\mathfrak f}}
 \newcommand\fg{{\mathfrak g}}
 \newcommand\fgl{\mathfrak{gl}}
 \newcommand\fh{{\mathfrak h}}

 \newcommand\fk{{\mathfrak k}}

 \newcommand\fp{{\mathfrak p}}
 
 \newcommand\fr{{\mathfrak r}}
 \newcommand\fs{{\mathfrak s}}
 \newcommand\fsl{\mathfrak{sl}}
 \newcommand\fso{\mathfrak{so}}
 \newcommand\fsp{\mathfrak{sp}}
 \newcommand\ft{{\mathfrak t}}
 \newcommand\fu{{\mathfrak u}}

 \newcommand\fz{{\mathfrak z}}

 \newcommand\fS{{\mathfrak S}}
 
 \newcommand\fU{{\mathfrak U}}

 \newcommand\cA{{\mathcal A}}

 \newcommand\bbC{{\mathbb C}}

 \newcommand\bbH{{\mathbb H}}

 \newcommand\bbP{{\mathbb P}}
 
 \newcommand\bbR{{\mathbb R}}
 \newcommand\bbS{{\mathbb S}}

 \newcommand\bbV{{\mathbb V}}
 \newcommand\bbW{{\mathbb W}}

 \newcommand\bbZ{{\mathbb Z}}

 \newcommand\tspan{\mathrm{span}}

 \newcommand\Ben{\begin{enumerate}}
 \newcommand\Een{\end{enumerate}}
 \newcommand\Bex{\begin{example}}
 \newcommand\Eex{\end{example}}

 \newcommand\ra{\rightarrow}

 \def\inj{\hookrightarrow}

 \newcommand\GL{\operatorname{GL}}

 \newcommand\CO{\mathrm{CO}}
 \newcommand\SO{\mathrm{SO}}
 \newcommand\SU{\text{SU}}
 \newcommand\Sp{\mathrm{Sp}}
  \newcommand\fco{\mathfrak{co}}
 \newcommand\fsu{\mathfrak{su}}
 \newcommand\Aut{\text{Aut}}
 \def\assoc/{associative}
 \def\arb/{arbitrary}
 \def\btw/{between}
 \def\coeff/{coefficient}
 \def\cohom/{cohomology}
 \def\coord/{coordinate}
 \def\coordsys/{coordinate system}
 \def\cpt/{compact}
 \def\cred/{completely reducible}
 \def\cts/{continuous}
 \def\dga/{differential-graded algebra}
 \def\dR/{de Rham}
 \def\Euc/{Euclidean} 
 \def\grp/{group}
 \def\hom/{homomorphism}
 \def\inv/{invariant}
 \def\iso/{isomorphism}
 \def\La/{Lie algebra}
 \def\Lag/{Lagrangian Grassmannian}
 \def\LG/{Lie group}
 \def\MA/{Monge--Amp\`ere}
 \def\MC/{Maurer--Cartan}
 \def\lintr/{linear transformation} 
 \def\mfld/{manifold}
 \def\nb/{normal bundle}
 \def\nbd/{neighbourhood}
 \def\nondeg/{non-degenerate}
 \def\posdef/{positive definite}
 \def\pu/{partition of unity}
 \def\rep/{representation}
 \def\Riem/{Riemannian}
 \def\sg/{subgroup}
 \def\ss/{semi-simple}
 \def\inv/{invariant}
 \def\irr/{irreducible}
 \def\Jacid/{Jacobi identity}
 \def\li/{linearly independent}
 \def\nd/{nowhere dependent}
 \def\nz/{nowhere zero}
 \def\on/{orthonormal}
 \def\onb/{\on/ basis}
 \def\orc/{\orth/ complement}
 \def\orth/{orthogonal}
 \def\orp/{\orth/ projection}
 \def\pde/{partial differential equation}
 \def\resp/{respectively}
 \def\seq/{sequence}
 \def\std/{standard}
 \def\SW/{Stiefel-Whitney}
 \def\uc/{universal cover}
 \def\vb/{vector bundle}
 \def\vf/{vector field}
 \def\vs/{vector space}
 \def\wrt/{with respect to}
 
 \renewcommand\mod{\,{\rm mod}\ }
 \newcommand\qbox[1]{\quad\mbox{#1}\quad}
 \renewcommand\dim{{\rm dim}}
 
\newtheorem{theorem}{Theorem}[section]

\newtheorem{prop}[theorem]{Proposition}
\theoremstyle{defn}
\newtheorem{defn}[theorem]{Definition}
\newtheorem{example}[theorem]{Example}

\newtheorem{recipe}[theorem]{Recipe}
\theoremstyle{remark}

\numberwithin{equation}{section}

 \newcommand\Weyl{\bbW}
 \newcommand\finf{\mathfrak{inf}}
 \newcommand\metric{{\rm g}}
 \newcommand\rkg{\ell}
 \newcommand\opn[1]{\operatorname{#1}}

 \renewcommand\myscale{0.8}

 \newcommand\yngtf[1]{\left. \yng(#1) \right._0}

\begin{document}

 \title[Maximally degenerate Weyl tensors in Riemannian and Lorentzian signatures]{Maximally degenerate Weyl tensors in Riemannian and Lorentzian signatures}

 \begin{abstract}
 We establish the submaximal symmetry dimension for Riemannian and Lorentzian conformal structures.  The proof is based on enumerating all subalgebras of orthogonal Lie algebras of sufficiently large dimension and verifying if they stabilize a non-zero Weyl tensor up to scale. Our main technical tools include Dynkin's classification of maximal subalgebras in complex simple Lie algebras, a theorem of Mostow, and Kostant's Bott--Borel--Weil theorem.
  \end{abstract}


 \author{Boris Doubrov}
 \address{Belarussian State University, Nezvisimosti Ave. 4, Minsk 220030, Belarus}
 \email{doubrov@islc.org}

 \author{Dennis The}
 \address{Mathematical Sciences Institute, Australian National University, ACT 0200, Australia}
 \email{dennis.the@anu.edu.au}

 \date{\today}

 \maketitle

 \section{Introduction}
 
 It is well-known that any (connected) conformal manifold $(M^n,[\metric])$ of signature $(p,q)$ in dimension $n \geq 3$ has Lie algebra of (infinitesimal) conformal symmetries $\finf([\metric])$ with dimension no greater than $\dim(\fso(p+1,q+1)) = \binom{n+2}{2}$.  Indeed, equality is realized if and only if $(M,[\metric])$ is locally conformally flat, so a natural question is: {\em Among all (connected) conformal manifolds $(M^n,[\metric])$ which are {\bf not} locally conformally flat, what is the maximal dimension $\fS(n)$ of $\finf([\metric])$?}  This is referred to as the gap problem, and $\fS(n)$ is the submaximal symmetry dimension.

The fundamental local invariant of conformal structures is the Weyl tensor in dimensions $\ge 4$, or
the Cotton--York tensor in dimension 3.   (In dimensions 1 and 2, any metric is locally conformally flat, and does not possess any local invariants).  The conformal structure is locally flat if and only if this tensor vanishes identically. 
The gap problem for conformal structures in arbitrary signature is directly related to the representation theory of real pseudo-orthogonal groups, or more precisely, to the study of orbits of minimal dimension in the space of (algebraic) Weyl (Cotton--York) tensors.
 
 B.~Kruglikov and D.~The recently studied the gap problem in the general context of parabolic geometries, and gave a universal upper bound $\fS \leq \fU$, where $\fU$ is algebraically determined \cite[Thm.~4.2.5]{KT2013}.  For conformal geometries when $n \geq 4$, we have (by ``prolongation-rigidity'' \cite{KT2013}):
 \[
 \fU(n) = n + \max\{ \dim(\fann(\phi)) \mid 0 \neq \phi \in \Weyl \},
 \]
 where $\Weyl$ is the space of (algebraic) Weyl tensors (as $(3,1)$-tensors), which is a representation of the conformal group $G_0 = \CO(p,q)$, and $\fann(\phi) \subset \fg_0 = \fco(p,q) = \bbR \times \fso(p,q)$ is the annihilator of a {\em nonzero} element $\phi \in \Weyl$.  In all non-Riemannian and non-Lorentzian signatures, null 2-planes exist (in the standard $\fg_0$-representation) and are responsible for the equality $\fS(n) = \fU(n) = \binom{n-1}{2} + 6$.   For $n=3$, defining $\fU(3)$ analogously via Cotton--York tensors, we have $\fS(3) = \fU(3) = 4$ in the Riemannian case, and $4 = \fS(3) < \fU(3) = 5$ in the Lorentzian case.  In this article, we settle the gap problem for Riemannian and Lorentzian conformal structures when $n \geq 4$.  We will prove:
 
 \begin{theorem} \label{T:main-thm} Let $n \geq 4$.  For Riemannian and Lorentzian conformal structures, $\fS(n) = \fU(n)$, with
 \[
 \begin{array}{c|c|c}
 & \mbox{Riemannian} & \mbox{Lorentzian}\\ \hline
 \fS(n) & \l\{ \begin{array}{cl} \binom{n-1}{2} + 3, & n = 5 \mbox{ or } n \geq 7;\\[0.02in] \frac{n^2}{4} + n, & n=4 \mbox{ or } 6\end{array}\r. & \binom{n-1}{2} + 4
 \end{array}
 \]
 \end{theorem}
 
 In the Riemannian case, the product of spheres (with their round metrics) $S^2 \times S^{n-2} \cong \frac{\SO(3) \times \SO(n-1)}{\SO(2) \times \SO(n-2)}$, $n \geq 5$ and complex projective space (with Fubini--Study metric) $\bbC\bbP^\rkg \cong \frac{\SU(\rkg+1)}{{\rm S}({\rm U}(1){\rm U}(\rkg))}$, $\rkg = \frac{n}{2} \geq 2$, are not conformally flat.  The corresponding conformal symmetry algebras have dimensions $\binom{n-1}{2} + 3$ and $\rkg^2+2\rkg = \frac{n(n + 4)}{4}$ respectively.  (In fact, all conformal symmetries for these metrics are Killing vector fields.)  For the Lorentzian case, consider the pp-wave
 \[
 \metric_{\rm pp}^{(3,1)} = dy^2 + dz^2 + dw dx + y^2 dw^2.
 \]
 When $n \geq 4$, taking the product $\metric = \metric_{\rm pp}^{(3,1)} + \metric_{\rm euc}^{(n-4,0)}$ with the flat Euclidean metric $\metric_{\rm euc}^{(n-4,0)} = \sum_{i=1}^{n-4} (du^i)^2$ yields a homogeneous non-conformally flat metric with $\binom{n-1}{2} + 4$ linearly independent conformal symmetries \cite{KT2013} (all but one of which are Killing vector fields).  Thus, the numbers $c(n)$ stated in Theorem \ref{T:main-thm} are lower bounds for $\fS(n)$, and since $\fS(n) \leq \fU(n)$ it suffices to show that $\fU(n) = c(n)$.  By homogeneity, at any given point of the models given above, the stabilizer subalgebras must annihilate a nonzero Weyl tensor.  With respect to a chosen basis, these are subalgebras of $\fg_0 = \fco(p,q)$ with dimension $c_0(n) := c(n) - n$.

 As remarked in \cite{KT2013}, establishing $\fS(n)$ in the conformal Riemannian case reduces to studying the maximal {\em isometries} for Riemannian metrics which are not conformally flat.  Several articles due to I.P.~Egorov \cite{Egorov1955,Egorov1956,Egorov1962} contributed to our understanding of this latter problem.  However, Egorov's final result in \cite{Egorov1962} is incomplete.\footnote{Egorov \cite{Egorov1962} asserts ``{\em The maximal order of the group of motions of a nonconformally euclidean space is exactly $\frac{(n-1)(n-2)}{2} + 5$.  This number can be replaced by $\frac{(n-1)(n-2)}{2} + 3$ in case the metric form is definite.''}  In modern terms, ``group of motions'' refers to a ``local group of isometries'', whose study is equivalent to that of isometry algebras.  In the Riemannian case, the $n=4$ and $n=6$ exceptions are omitted, although Egorov was certainly aware of the $n=4$ case \cite{Egorov1955}.  In the Lorentzian case, the result is false by our (conformal) symmetry bound proven in this article.}  Alternatively, a classification of Riemannian spaces with abundant isometries was given in \cite{KN1972}, from which one can conclude that any space with an isometry {\em group} of dimension greater than $c(n)$ is (locally) conformally flat.  However, since this was a global classification, we cannot immediately deduce $\fS(n)$ which concerns the (infinitesimal) symmetry {\em algebra}.\footnote{In this article, we study (conformal) symmetry algebras.  The distinction between symmetry groups versus symmetry algebras is important.  Given a subalgebra $\fk\subset \fg$, there may not exist a global model $G/K$ when there exist subgroups $K\subset G$ corresponding to $\fk$ which are not closed. The smallest example of this phenomenon appears in $\dim(\fg/\fk)=5$ (see \cite{Mostow1950}): take $\fg=\fsu(2) \times \fsu(2)$, $\fk =\{ \opn{diag}(ix,-ix) \times \opn{diag}(i\alpha x, -i\alpha x) \mid x \in \bbR \}$, where $\alpha$ is irrational. In this case, we can construct a 5-dimensional (Riemannian) manifold with a 6-dimensional Lie algebra of isometries, but there is no completion to get the corresponding 6-dimensional isometry group.}  Our proof for $\fS(n)$ is independent of earlier papers on this topic and is purely algebraic in both Riemannian and Lorentzian signatures. See also~\cite{KM2013} for the submaximal dimensions of projective and affine symmetries of pseudo-Riemannian metrics.

 In \cite{KT2013}, we saw that the $(\fg_0)_\bbC = \fco(n,\bbC)$ action on $\Weyl_\bbC = \Weyl \otimes_\bbR \bbC$ has maximal (proper) annihilators realized precisely on the unique (closed) orbit of the lowest weight line.  All such maximal annihilators are conjugate, have dimension $\binom{n-1}{2} + 6 - n = \binom{n-2}{2} + 4$, and any such is the stabilizer subalgebra of a null 2-plane in the standard representation.
 
The $\bbR^*$-factor in $G_0 = \CO(p,q) = \bbR^* \times \SO(p,q)$ acts {\em non-trivially} on $\Weyl$, so $\fk = \fann(\phi) \subset \fg_0$ is uniquely determined by (and has the same dimension as) some $\fk' \subset \fso(p,q)$ which preserves a nonzero $\phi \in \Weyl$ {\em up to scale}.  We will use the following terminology:

 \begin{defn} A subalgebra $\fk \subset \fso(p,q)$ is {\em visible} if there exists $0 \neq \phi \in \Weyl$ such that
 \[
 \fk = \fco(\phi) := \{ X \in \fso(p,q) \mid X \cdot \phi = const \cdot \phi \}.
 \]
 We say that $\phi$ is {\em maximally degenerate} if $\fco(\phi)$ has maximal dimension among all visible subalgebras.
 \end{defn}

 Using $\fso(p,q)$, we regard $\Weyl$ as totally trace-free $(4,0)$-tensors with well-known index symmetries:
 \[
 \phi_{abcd} = -\phi_{bacd} = -\phi_{abdc}, \qquad \phi_{abcd} = \phi_{cdab}, \qquad \phi_{a[bcd]} = 0.
 \]
 The complexification $\Weyl_\bbC$ is an $\fso(n,\bbC)$-representation which is irreducible for $n \geq 5$ with highest weight $2\lambda_2$ when $n \geq 7$, $2\lambda_2 + 2\lambda_3$ when $n=6$, and $4\lambda_2$ when $n=5$ in terms of the fundamental weights $\lambda_i$ of $\fso(n,\bbC)$.  When $n=4$, $\Weyl_\bbC$ decomposes into two irreducible modules with highest weights $4\lambda_1$ and $4\lambda_1'$.  In each case, the (set of) weights is stable with respect to any symmetry of the Dynkin diagram of $\fso(n,\bbC)$, so any (inner or outer) automorphism $f: \fso(n,\bbC) \to \fso(n,\bbC)$ can be extended to an automorphism $F: \Weyl_\bbC \to \Weyl_\bbC$, i.e.
 \[
 f(x) \cdot F(\phi) = F(x\cdot\phi), \qquad \forall x \in \fso(n,\bbC), \quad \forall \phi \in \Weyl_\bbC.
 \]
 Thus, the class of visible subalgebras of $\fso(n,\bbC)$ is stable with respect to the full group of automorphisms.  This is especially important for the $n=8$ case, where the subtle notion of triality arises \cite{Harvey1990}.

 In non-Riemannian and non-Lorentzian signatures, a maximally degenerate Weyl tensor is obtained as follows: taking a basis $\{ \omega^i \}$ of $(\bbR^{p,q})^*$ with respect to which $E = \tspan\{ \omega^1, \omega^2 \}$ is null, we have
 \[
 \phi = (\omega^1 \wedge \omega^2)^2,
 \]
 and $\fco(\phi)$ is the parabolic subalgebra $\fp_2 \subset \fso(p,q)$ which stabilizes $E$. Conversely, any maximally degenerate Weyl tensor can be expressed this way.
 
In Riemannian and Lorentzian signatures, the dimension $\binom{n-2}{2} + 4$ (of $\fp_2$ above) is not realizable by a visible subalgebra since null 2-planes do not exist, and the form of a maximally degenerate Weyl tensor is more complicated -- see \eqref{E:W-Riem1} for $n \neq 4, 6$, \eqref{E:W-Riem2} for $\rkg=2,3,4$, and \eqref{E:W-Lor}. 
 In these signatures, we refer to a subalgebra of $\fso(p,q)$ as {\em inadmissible} if its dimension is strictly less than $c_0(n)$ or strictly greater than $\binom{n-2}{2} + 3$.  We will show that no visible subalgebra has dimension $d$ in the range $c_0(n) +1 \leq d \leq \binom{n-2}{2} + 3$. To show this we enumerate all possible subalgebras of $\fso(p,q)$ in this dimension range and verify that they do not preserve any Weyl tensor up to scale.

The enumeration of subalgebras is based on the classification of reductive Lie algebras in complex simple Lie algebras due to Dynkin~\cite{Dynkin1952a,Dynkin1952b}. In the Riemannian case, the transition to the complex field is straightforward, as any subalgebra of $\fso(n)$ is reductive. Conversely, any semisimple subalgebra of $\fso(n,\bbC)$ admits a unique real form in $\fso(n)$. (Analyzing the central part of reductive subalgebras requires more care, but is also not difficult). In the Lorentzian case, we have to deal with the parabolic subalgebra $\fp_1 \subset \fso(1,n-1)$ and its subalgebras. The branching of $\Weyl$ to these subalgebras can be analyzed using Kostant's Bott--Borel--Weil theorem~\cite{Kos1961, BE1989, CS2009}.  To summarize:

 \begin{theorem} \label{T:subalg-classify} For $n \geq 4$, we have the following classification of maximal visible subalgebras of $\fso(p,q)$ up to automorphisms\footnote{Although $\fu(4)$ and $\fso(2) \times \fso(6)$ are not conjugate in $\fso(8)$, they are equivalent by triality via an outer automorphism of $\fso(8)$.}:
 \[
 \begin{array}{|c|c|c|c|} \hline
 \fso(p,q) & \mbox{Maximal visible subalgebras} & \mbox{Dimension} & \mbox{Range}\\ \hline\hline
 \fso(n) &  
 \begin{array}{lll} \fso(2) \times \fso(n-2) \subset \fso(n)\\
 \fu(2) \subset \fso(5)\\
 \fu(\rkg) \subset \fso(2\rkg) \\ 
 \end{array} & \begin{array}{c} \binom{n-2}{2} + 1 \\ 4\\ \rkg^2 \end{array} & \begin{array}{l} n=5 \mbox{ or } \geq 7\\ n=5\\ n=2\rkg=4 \mbox{ or } 6\end{array}\\ \hline
 \fso(1,n-1) & (\bbR \op \fso(n-3)) \ltimes \bbR^{n-2} \subset \fp_1 & \binom{n-2}{2} + 2 & n \geq 4\\[0.02in] \hline
 \begin{array}{c} \fso(p,q)\\ (p,q \geq 2) \end{array} & \fp_2 & \binom{n-2}{2} + 4 & n \geq 4\\ \hline
 \end{array}
 \]
 \end{theorem}

A posteriori, we deduce that the maximal visible subalgebra in $\fso(1,n-1)$ stabilizes a degenerate (but not totally null) 2-plane in $\bbR^{1,n-1}$.

The techniques used in the paper for enumerating large subalgebras of (real and complex) semisimple Lie algebras are very generic and can be used to tackle other problems of a similar nature. Whenever appropriate, we provide links to more advanced results in this area, even though they are not directly needed for our main result.

 In Section \ref{S:subalg}, we review Dynkin's classification of subalgebras of complex simple Lie algebras.  We use this in Section \ref{S:Riemannian} to prove Theorem \ref{T:main-thm} in the Riemannian case.  The Lorentzian case is resolved in Section \ref{S:Lorentzian} with the exception of the $n=6$ case, for which it remains to demonstrate that there is no Lorentzian analogue of the inclusion $\fgl(3,\bbC) \subset \fso(6,\bbC)$.  This is proven in Section \ref{S:real}, where we study the more general problem of classifying all real forms of a given reductive subalgebra in a complex simple Lie algebra, focusing in particular on the case of $\fgl(\rkg,\bbC) \subset \fso(2\rkg,\bbC)$.\\
  
 {\bf Conventions:} We use the standard labelling $A_\rkg = \fsl(\rkg+1,\bbC$), $B_\rkg = \fso(2\rkg+1,\bbC)$, $C_\rkg = \fsp(2\rkg,\bbC)$, $D_\rkg = \fso(2\rkg,\bbC)$, and $E_6,E_7,E_8,F_4,G_2$ for complex simple Lie algebras, and the Bourbaki ordering of simple roots.  We recall the special isomorphisms $\fso(3,\bbC) = B_1 = A_1 = \fsl(2,\bbC)$,  $\fso(4,\bbC) = D_2 = A_1 \times A_1$, and $\fso(6,\bbC) = D_3 = A_3 = \fsl(4,\bbC)$.  Our symmetrizers and antisymmetrizers are projection operators, e.g. $ab = \frac{1}{2} (a \otimes b + b \otimes a)$.\\
 
 {\bf Acknowledgements:} We are grateful for conversations with B. Kruglikov.  R.L. Bryant also gave useful insights on
 \href{http://mathoverflow.net/questions/105995/most-degenerate-weyl-tensors-in-riemannian-and-lorentzian-signature}{MathOverflow}.  D.T. was partially supported by a Research Fellowship from the Australian Research Council. We would also like to thank the Erwin Schr\"odinger Institute, where the essential part of this research was carried out. 
 
 \section{Subalgebras of complex simple Lie algebras}
 \label{S:subalg}
 
 \subsection{Dynkin's subalgebra classification}
 \label{S:Dynkin-list}
 
 In \cite{Dynkin1952a}, Dynkin classified all maximal\footnote{A subalgebra $\fk \subset \fg$ is {\em maximal} if it is not strictly contained in any proper subalgebra of $\fg$.} subalgebras $\ff$ of a complex classical Lie algebra $\fg$.  (See also \cite[Chp.\ 6, Thms.\ 3.1--3.3]{OV1994}).  Up to conjugacy, there are three possibilities for $\ff \subset \fg$:
 \begin{enumerate}
 \item {\em reducible}, i.e. there is a proper $\ff$-invariant subspace in the standard $\fg$-representation.  Then:
 \begin{enumerate}
 \item $\ff$ is a maximal parabolic subalgebra of $\fg$, or
 \item $\fg = \fso(n,\bbC)$: $\ff \cong \fso(k,\bbC) \times \fso(n-k,\bbC)$ for $0 < k < n$, or
 \item $\fg = \fsp(n,\bbC)$: $\ff \cong \fsp(k,\bbC) \times \fsp(n-k,\bbC)$ for $k,n$ both even, $0 < k < n$.
 \end{enumerate}
 \item {\em irreducible non-simple}: All possibilities are given in Table \ref{F:S-ns}.
 \begin{table}[h]
 $\begin{array}{|c|c|l|} \hline
 \fg & \ff & \mbox{Restrictions on $d_i = \dim(V_i)$}\\ \hline
 \fsl(V_1 \ot V_2) & \fsl(V_1) \times \fsl(V_2) & d_i \geq 2\\
 \fso(V_1 \ot V_2) & \fso(V_1) \times \fso(V_2) & 4 \neq d_i \geq 3\\
 \fso(V_1 \ot V_2) & \fsp(V_1) \times \fsp(V_2) & d_i \geq 2;\, (d_1,d_2) \neq (2,2)\\
 \fsp(V_1 \ot V_2) & \fsp(V_1) \times \fso(V_2) & d_1 \geq 2;\, 4 \neq d_2 \geq 3 \mbox{ or } (d_1,d_2) = (2,4)\\ \hline
 \end{array}$
 \caption{Maximal irreducible non-simple subalgebras $\ff$ of a classical complex simple Lie algebra $\fg$}
 \label{F:S-ns}
 \end{table}

 \item {\em irreducible simple}: Given $\ff$ simple, let $\psi : \ff \to \fgl(\bbV)$ be a (nontrivial) irreducible representation (irrep). Then $\psi$ is {\em self-dual} if there is a nondegenerate $\ff$-invariant bilinear form on $\bbV$.  Aside from the exceptions\footnote{While \cite[Table 7]{OV1994} lists $B_3 = \mathfrak{spin}(7,\bbC) \to \fso(\bbV_{\lambda_3}) \cong \fso(8,\bbC)$ as non-maximal, this is a typo.} listed in 
\cite[Table 7]{OV1994}, we always have:
 \begin{enumerate}
 \item If $\psi$ is not self-dual, then $\psi(\ff) \subset \fsl(\bbV)$ is maximal;
 \item If $\psi$ is orthogonal, then $\psi(\ff) \subset \fso(\bbV)$ is maximal;
 \item If $\psi$ is symplectic, then $\psi(\ff) \subset \fso(\bbV)$ is maximal.
 \end{enumerate}

 \end{enumerate}
 
 In the reducible case, consider a maximal parabolic subalgebra $\fp \subset \fg$.  Up to conjugacy, these are in 1-1 correspondence with Dynkin diagrams marked by a cross on a single node $k$ (corresponding to the simple root $\alpha_k$), and we denote the corresponding parabolic subalgebra $\fp_k$.  Any maximal {\em reductive} subalgebra of a (maximal) parabolic subalgebra is conjugate to its {\em Levi factor}~\cite{Mostow1956}.  There is a simple recipe to determine its structure:
 
 \begin{framed}
 \begin{recipe}
 The Levi factor of a parabolic subalgebra has center with dimension equal to the number of crosses on the corresponding marked Dynkin diagram.  Removing all crosses yields the Dynkin diagram of the semisimple part of the Levi factor.
 \end{recipe}
 \end{framed}

 \begin{example}
 The Levi factors for the maximal parabolics in $B_\rkg$ and $D_\rkg$ are given in Table \ref{F:Levi}.
 \begin{table}[h]
 \[
 \begin{array}{|c|c|c|c|} \hline
 \fg & k & \mbox{Levi factor} & \mbox{Dimension}\\ \hline\hline
 B_\rkg & 1 & \bbC \times B_{\rkg-1} & \binom{2\rkg-1}{2} + 1\\
 & 2 \leq k \leq \rkg-1 & \bbC \times A_{k-1} \times B_{\rkg-k} & k^2 + \binom{2\rkg-2k+1}{2}\\
 & \rkg & \bbC \times A_{\rkg-1} & \rkg^2\\ \hline
 D_\rkg & 1 & \bbC \times D_{\rkg-1} & \binom{2\rkg-2}{2} + 1\\
 & 2 \leq k \leq \rkg-2 & \bbC \times A_{k-1} \times D_{\rkg-k} & k^2 + \binom{2\rkg-2k}{2}\\
 & \rkg-1 \mbox{ or } \rkg & \bbC \times A_{\rkg-1} & \rkg^2\\ \hline
 \end{array}
 \]
 \caption{Levi factors of maximal parabolic subalgebras of $B_\rkg$ and $D_\rkg$}
 \label{F:Levi}
 \end{table}
 \end{example}
 
 In the irreducible simple case, any irrep $\psi : \ff \to \fgl(\bbV)$ is generated from its unique dominant integral (highest) weight $\lambda$.  Writing $\lambda = \sum_i r_i \lambda_i$ in terms of fundamental weights $\{ \lambda_i \}$ for $\ff$, we encode $\lambda$ on the Dynkin diagram of $\ff$ by putting the coefficients $r_i$ above the corresponding nodes.  Then $\psi$ is self-dual if and only if $\lambda$ is invariant under the {\em duality involution}.  The duality involution is trivial (so all irreps are self-dual) except for $A_\rkg$ for $\rkg \geq 2$, $D_\rkg$ for $\rkg$ odd, and $E_6$, where it is the unique non-trivial Dynkin diagram automorphism.
 
 Given a self-dual irrep $\psi : \ff \to \fgl(\bbV_\lambda)$, we can determine if it is orthogonal or symplectic as follows \cite[Corollary on page 98]{OV1994}.  If $\ff$ does not appear in Table \ref{F:self-dual}, then the irrep is orthogonal.\footnote{Since we take the Bourbaki ordering, the coefficients in the $E_7$ case are a permutation of those appearing in \cite{OV1994}.}  Otherwise, writing $\lambda = \sum_i r_i \lambda_i$, we examine the parity of the linear combination of the coefficients indicated in the table.  The irrep is orthogonal if the parity is even and symplectic if it is odd.

 \begin{table}[h]
 \[
 \begin{array}{|c|c|c|c|c|} \hline
 A_{4k+1} & B_\rkg\, (\rkg \equiv 1,2\, \mod 4) & C_\rkg & D_{4k+2} & E_7\\ \hline\hline
 r_{2k+1} & r_\rkg & \sum_{i \,\,{\rm odd}} r_i & r_{4k+1} + r_{4k+2} & r_1 + r_5 + r_7\\[0.02in] \hline
 \end{array}
 \]
 \caption{Data to determine if a self-dual irrep $\psi : \ff \to \fgl(\bbV_\lambda)$ is orthogonal or symplectic}
 \label{F:self-dual}
 \end{table}

 \subsection{Regular reductive subalgebras}
 \label{S:regular}
 
 Dynkin's classification in Section \ref{S:Dynkin-list} will be mostly sufficient for our purposes, but we will need some facts about reductive subalgebras of $\fg_2$ (Example \ref{E:g2-red-subalg}).  For this purpose, we recall an alternative classification of subalgebras of an {\em arbitrary} complex semisimple Lie algebra $\fg$, also introduced by Dynkin.  Our presentation here summarizes that in \cite[Chp.\ 6]{OV1994}.
 A subalgebra $\ff \subset \fg$ is {\em regular} if there is some Cartan subalgebra $\fh \subset \fg$ such that $[\fh,\ff] \subset \ff$. In this case, letting $\Delta:= \Delta_\fg(\fh)  \subset \fh^*$ be the corresponding root system, there exists a subspace $\ft \subset \fh$ and a subsystem $\Gamma \subset \Delta$ such that
 \begin{align} \label{E:f-reg-subalg}
 \ff = \ff(\ft,\Gamma) = \ft \op \bop_{\alpha \in \Gamma} \fg_\alpha,
 \end{align}
 where $\fg_\alpha$ is the root space corresponding to $\alpha$.  Conversely, given $\ft \subset \fh$ and $\Gamma \subset \Delta$, it follows that $\ff(\ft,\Gamma)$ is a regular subalgebra iff: (i) $\Gamma \subset \Delta$ is {\em closed}, i.e. $\alpha, \beta \in \Gamma$ and $\alpha + \beta \in \Delta$, then $\alpha + \beta \in \Gamma$; and (ii) if $\{ \pm \alpha \} \subset \Gamma$, then $[\fg_\alpha, \fg_{-\alpha}] \subset \ft$.  It is moreover reductive iff $\Gamma$ is {\em symmetric}, i.e. $\alpha \in \Gamma$ iff $-\alpha \in \Gamma$.  In the reductive case, $\ff(\ft,\Gamma)$ is semisimple iff $\ft = \tspan\{ [\fg_\alpha,\fg_{-\alpha}] \mid \alpha \in \Gamma \}$, which makes $\ft \subset \ff$ a Cartan subalgebra, and $\Gamma = \Delta_\ff(\ft)$ a root system.  If $\Gamma \subset \Delta$ is closed and symmetric, define the semisimple subalgebra $\ff(\Gamma) := \ff(\ft,\Gamma)$, where $\ft = \tspan\{ [\fg_\alpha,\fg_{-\alpha}] \mid \alpha \in \Gamma \}$.
 
  \begin{example} \label{E:ab-ss} Let $\fa \subset \fg$ be an abelian subalgebra consisting of semisimple elements.  Then the centralizer $\fz(\fa) \subset \fg$ of $\fa$ is a regular reductive subalgebra.  In particular, any maximal reductive subalgebra with non-trivial centre is regular.
 \end{example}
 
 \begin{example}
 Let $\fg$ be a complex semisimple Lie algebra and $\fp$ a parabolic subalgebra, which induces a $\bbZ$-grading $\fg = \fg_- \op \fg_0 \op \fg_+$.  Then $\Delta(\fg_0) \subset \Delta$ is a closed symmetric subsystem.
 \end{example}
  
Let $N \subset \Delta$.  Define $[N] := \{ \alpha \in \Delta \mid \alpha \in \tspan_\bbZ(N) \}$.  Then $N$ is a {\em $\pi$-system} if $N$ is linearly independent and $\alpha - \beta \not\in \Delta$ for all $\alpha,\beta\in N$.
 If $\Gamma \subset \Delta$ is closed and symmetric, it is a root system of $\ff(\Gamma)$, so $\Gamma = [N]$ for some simple roots $N \subset\Gamma$, which is necessarily a $\pi$-system.  Conversely, any $\pi$-system determines a closed symmetric subsystem.  Thus, the following objects are equivalent:
 \Ben
 \item maximal rank (i.e. $\ft = \fh$) regular reductive subalgebras $\ff(\fh,\Gamma) \subset \fg$;
 \item regular semisimple subalgebras $\ff(\Gamma) \subset \fg$;
 \item closed symmetric subsystems $\Gamma \subset \Delta$;
 \item $\pi$-systems $N \subset \Delta$.
 \Een  

There is an efficient classification \cite{GG1978}, \cite[Chp.~6, Thm.~1.2]{OV1994} of maximal proper $\pi$-systems $N \subset \Delta$.   If $\fg$ is semisimple but not simple, then by the Corollary in \cite[Chp.~6, Prop.~1.1]{OV1994}, all regular semisimple subalgebras of $\fg$ are products of such subalgebras in each simple ideal of $\fg$.
So suppose $\fg$ is simple, i.e. $\Delta$ is indecomposable, and choose simple roots $\Pi = \{ \alpha_1, ..., \alpha_\rkg \} \subset \Delta$.  For any $\alpha \in \Delta$, write $\alpha = \sum_i m_i(\alpha) \alpha_i$.  Let $\widetilde\alpha \in \Delta$ be the highest root, with coefficients $n_i = m_i(\widetilde\alpha)$.  Let $\alpha_0 = -\widetilde\alpha$ be the lowest root, and $\widetilde\Pi = \{ \alpha_0 \} \cup \Pi$ the extended system of simple roots.

\begin{framed}
\begin{recipe} \label{R:reg-reductive}
 All maximal regular reductive subalgebras of a complex simple Lie algebra $\fg$ correspond to a $\pi$-system $N \subset \Delta$ of type I or II, defined as follows:  For $1 \leq k \leq \rkg = \rnk(\fg)$,
 \Ben
 \item[I:] $n_k = 1$.  Take $N = \Pi \backslash \{ \alpha_k \}$, so $\Gamma = \{ \alpha \in \Delta \mid  m_k(\alpha) = 0 \}$.  Then $\rnk(\Gamma) = \rnk(\Delta) - 1$.
 \item[II:] $n_k > 1$ and $p | n_k$ for some $p > 1$ prime.  Take $N = \widetilde\Pi \backslash \{ \alpha_k \}$, and $\Gamma = \{ \alpha\in \Delta \mid m_k(\alpha) \equiv 0 \mod p \}$.  Then $\rnk(\Gamma) = \rnk(\Delta)$.
 \Een
 \end{recipe}
 \end{framed}
 
By \cite[Chp.~6, Prop.~1.5]{OV1994}, the classification in Recipe \ref{R:reg-reductive} is up to conjugacy, and from this we can deduce the classification up to any subgroup lying between the group of inner automorphisms $\opn{Int}(\fg)$ and the full group of automorphisms $\Aut(\fg)$.
 
 \begin{example}
 There are two maximal regular reductive subalgebras of $\fso(4,\bbC) \cong \fso(3,\bbC)_L \times \fso(3,\bbC)_R$, namely $\fso(3,\bbC)_L \times \fso(2,\bbC)_R \cong \fgl(2,\bbC)$ and $\fso(2,\bbC)_L \times \fso(3,\bbC)_R \cong \fgl(2,\bbC)$.\footnote{We have indicated subscripts $L,R$ to remind the reader that the embeddings of $\fso(3,\bbC)_L$, $\fso(3,\bbC)_R$ into $\fso(4,\bbC)$ are {\em irreducible}. This differs from the canonical reducible embedding $\fso(3,\bbC) \inj \fso(4,\bbC)$.}  Though these subalgebras are not conjugate, they are equivalent via an outer automorphism of $\fso(4,\bbC)$.
 \end{example}
 
 \begin{example} By Recipe \ref{R:reg-reductive} and Table \ref{F:BD-Dynkin}, we obtain the complete list of maximal regular reductive subalgebras of $B_\rkg = \fso(2\rkg+1,\bbC)$ and $D_\rkg = \fso(2\rkg,\bbC)$ in Table \ref{F:reg-subalg}.
 \end{example}
 
 \begin{center}
 \begin{table}[h]
 \begin{tabular}{|c|c|c|c|} \hline
 $B_2$ & $B_\rkg\,\, (\rkg \geq 3)$ \\ \hline
 \begin{tiny}
 \begin{tikzpicture}[scale=\myscale,baseline=-3pt]

 \dbond{r}{0,0};
 \dbond{l}{1,0};
 \DDnode{w}{0,0}{\alpha_1};
 \DDnode{w}{1,0}{\alpha_2};
 \DDnode{w}{2,0}{-\widetilde\alpha};

 \useasboundingbox (-.4,-.2) rectangle (2.4,0.55);
 \end{tikzpicture}
 \end{tiny} & \begin{tiny}
 \begin{tikzpicture}[scale=\myscale,baseline=-3pt]

 \bond{0,0};
 \bond{1,0};
 \tdots{2,0};
 \bond{3,0};
 \dbond{r}{4,0};
 \draw (1,0) -- +(-0.5,-0.865);
 \DDnode{w}{0.5,-0.865}{\hspace{-0.12in}-\widetilde\alpha};
 
  \foreach \x/\y in {0/1,1/2,2/3,3/ \rkg-2,4 / \rkg-1, 5 / \rkg}
	\DDnode{w}{\x,0}{\alpha_{\y}};

 \useasboundingbox (-.4,-1.2) rectangle (5.4,0.55);
 \end{tikzpicture}
 \end{tiny} \\ \hline
  $\widetilde\alpha = \alpha_1 + 2\alpha_2$ & $\widetilde\alpha = \alpha_1 + 2\alpha_2 + ... + 2\alpha_\rkg$ \\ \hline\hline
$D_3$ & $D_\rkg\,\, (\rkg \geq 4)$\\ \hline
 \begin{tiny}
 \begin{tikzpicture}[scale=\myscale,baseline=-3pt]

 \diagbond{u}{0,0};
 \diagbond{d}{0,0};
 \draw (1,0) -- +(-0.5,-0.865);
 \draw (1,0) -- +(-0.5,0.865);
 \DDnode{w}{1,0}{\,\,\,\,-\widetilde\alpha};
   
 \DDnode{w}{0,0}{\!\!\!\alpha_1};
 \DDnode{w}{0.5,0.865}{\alpha_2};
 \DDnode{w}{0.5,-0.865}{\alpha_3};

 \useasboundingbox (-.4,-.2) rectangle (1.4,0.55);
 \end{tikzpicture}
 \end{tiny} &
 \begin{tiny}
 \begin{tikzpicture}[scale=\myscale,baseline=-3pt]

 \bond{0,0};
 \bond{1,0};
 \tdots{2,0};
 \bond{3,0};
 \diagbond{u}{4,0};
 \diagbond{d}{4,0};
 \draw (1,0) -- +(-0.5,-0.865);
 \DDnode{w}{0.5,-0.865}{\hspace{-0.12in}-\widetilde\alpha};
  
  \foreach \x/\y in {0/1,1/2,2/3,3/ \rkg-3}
	\DDnode{w}{\x,0}{\alpha_{\y}};
 \DDnode{w}{4,0}{\alpha_{\rkg-2}};
 \DDnode{w}{4.5,0.865}{\alpha_{\rkg-1}};
 \DDnode{w}{4.5,-0.865}{\,\,\alpha_\rkg};

 \useasboundingbox (-.4,-1.2) rectangle (5.4,0.55);
 \end{tikzpicture}
 \end{tiny}\\ \hline
 $\widetilde\alpha = \alpha_1 + \alpha_2 + \alpha_3$ & $\widetilde\alpha = \alpha_1 + 2\alpha_2 + ... + 2\alpha_{\rkg-2} + \alpha_{\rkg-1} + \alpha_\rkg$\\ \hline
 \end{tabular}
 \caption{Extended Dynkin diagrams and highest roots for $B_\rkg$ and $D_\rkg$}
 \label{F:BD-Dynkin}
 \end{table}
 \end{center}
 \begin{center}
 \begin{table}[h]
 $\begin{array}{|c|c|c|c|c|} \hline
  & k & \mbox{Type} & \begin{tabular}{c} \mbox{Reductive}\\ \mbox{subalgebra} \end{tabular} & \mbox{Classical form}\\ \hline\hline
 B_\rkg & 1 & \mbox{I} & \bbC \times B_{\rkg-1} & \bbC \times \fso(2\rkg-1,\bbC)\\
  (\rkg \geq 2) & 2 \leq k \leq \rkg-1 & \mbox{II} & D_k \times B_{\rkg - k} & \fso(2k,\bbC) \times \fso(2(\rkg-k) + 1,\bbC)\\
 & k = \rkg & \mbox{II} & D_\rkg & \fso(2\rkg,\bbC)\\ \hline
 D_\rkg  & 1 & \mbox{I} & \bbC \times D_{\rkg-1} & \bbC \times \fso(2(\rkg-1),\bbC)\\
 (\rkg \geq 3)& \rkg-1, \rkg & \mbox{I} & \bbC \times A_{\rkg-1} & \fgl(\rkg,\bbC)\\
 & 2 \leq k \leq \rkg-2 & \mbox{II} & D_k \times D_{\rkg - k} & \fso(2k,\bbC) \times \fso(2(\rkg-k),\bbC)\\ \hline
 \end{array}$
 \caption{Maximal regular reductive subalgebras of $B_\rkg$ and $D_\rkg$}
 \label{F:reg-subalg}
 \end{table}
 \end{center}
 
  An $S$-subalgebra is a reductive\footnote{In fact, any $S$-subalgebra is semisimple.} subalgebra not contained in any proper regular subalgebra. Note that $\fso(k,\bbC) \times \fso(n-k,\bbC) \subset \fso(n,\bbC)$, $0 < k < n$, are always regular subalgebras when $n$ is odd.  When $n$ is even, it is regular when $k$ is even, and it is an $S$-subalgebra when $k$ is odd.
 
 \begin{example} \label{E:g2-red-subalg} The 14-dimensional (complex) simple Lie algebra $\fg_2$ has $\widetilde\alpha = 3\alpha_1 + 2\alpha_2$ and extended Dynkin diagram
  \begin{tiny}
 \begin{tikzpicture}[scale=\myscale,baseline=-3pt]

 \tbond{l}{0,0};
 \bond{1,0};
 \DDnode{w}{0,0}{\alpha_1};
 \DDnode{w}{1,0}{\alpha_2};
 \DDnode{w}{2,0}{-\widetilde\alpha};

 \useasboundingbox (-.4,-.2) rectangle (2.4,0.55);
 \end{tikzpicture}
 \end{tiny}.  There are two distinct $\pi$-systems, each of type II.  For $k=1$ and $k=2$, we get diagrams of type $A_2 = \fsl(3,\bbC)$ and $A_1 \times A_1 = \fsl(2,\bbC) \times\fsl(2,\bbC)$ respectively.  Thus, any regular reductive subalgebra is at most 8-dimensional.  By \cite[Thm.~3.4 \& 3.5 on pg.\ 207]{OV1994}, all $S$-subalgebras are at most 3-dimensional.  Thus, any reductive subalgebra is at most 8-dimensional.
 \end{example}

 \section{The Riemannian case}
 \label{S:Riemannian}
 
 We return now to our original problem.  Since $\SO(n)$ is compact, then every subalgebra of $\fso(n)$ is compact and reductive.  By compactness, every element of $\fso(n)$ acts on $\Weyl$ with either complex conjugate eigenvalues or the real eigenvalue 0.  Thus, any visible subalgebra $\fk \subset \fso(n)$ is of the form $\fk = \fann(\phi)$ for some nonzero $\phi \in \Weyl$.  There is a 1-1 correspondence between reductive subalgebras of $\fso(n)$ and reductive subalgebras of its complexification $\fso(n,\bbC)$ \cite[Chp.~4, Thm.~2.7]{OV1994}.  We will:
 \Ben
 \item Find all reductive subalgebras $\fk \subset \fso(n,\bbC)$ with $\dim(\fk)$ greater than 
 \[
 c_0(n) := \l\{ \begin{array}{cl} \binom{n-2}{2} + 1, & n = 5 \mbox{ or } n \geq 7;\\ \frac{n^2}{4}, & n=4 \mbox{ or } 6\end{array}\r..
 \]
 \item For each $\fk$ above, check that no nonzero Weyl tensor is annihilated.
 \Een
 The $n=4$ case is exceptional since $\fso(4) \cong \fso(3)_L \times \fso(3)_R$ is semisimple and not simple.  Any proper subalgebra of $\fso(4)$ has at most dimension $c_0(4) = 4$, and this bound is realized by $\dim(\fu(2))$, so $\fU(4) = 8$ (realized by $\bbC\bbP^2$).  We concentrate on the $n \geq 5$ case.

 We use Dynkin's subalgebra classification to study {\em reductive} subalgebras of $\fso(n,\bbC)$ with dimension at least $c_0(n)$, where $n = 2\rkg+1$ for $B_\rkg$ and $n = 2\rkg$ for $D_\rkg$.  Consider first the reducible case.  Among $\fso(k,\bbC) \times \fso(n-k,\bbC)$, these include only $\fso(n-1,\bbC)$ and $\fso(2,\bbC) \times \fso(n-2,\bbC)$, where the latter is excluded when $n=6$.  Up to automorphisms, the Levi factors in Table \ref{F:Levi} have maximal dimension $c_0(n)$, with equality realized by:
 \begin{align} \label{E:Riem-subalg}
 \left\{ \begin{array}{ll} \bbC \times \fso(n-2,\bbC) \subset \fso(n,\bbC), & n=5 \mbox{ or } n \geq 7;\\
 \fgl(\rkg,\bbC) \subset \fso(2\rkg,\bbC), & \rkg = 2,3; \\ \fgl(2,\bbC) \subset \fso(5,\bbC).
 \end{array} \right.
 \end{align}
Note that by triality, i.e. via an outer automorphism of $\fso(8,\bbC)$, we have $\fgl(4,\bbC) \cong \bbC \times \fso(6,\bbC)$.
 
 In the irreducible non-simple case (Table \ref{F:S-ns}), all subalgebras are inadmissible:
\begin{itemize}
 \item $\fso(V_1) \times \fso(V_2) \subset \fso(V_1 \otimes V_2)$ :  $n = d_1 d_2$ with $d_1,d_2 \geq 3$, and we may assume $d_1 \leq \sqrt{n}$, so
 \[
 \dim(\fso(V_1) \times \fso(V_2)) = \binom{d_1}{2} + \binom{d_2}{2} < \frac{d_1{}^2 + d_2{}^2}{2} \leq \frac{3^4 + n^2}{18} < \binom{n-2}{2} + 1 = c_0(n).
 \]
 \item $\fsp(V_1) \times \fsp(V_2) \subset \fso(V_1 \otimes V_2)$ :  
  Recall $\dim(\fsp_{2k}(\bbC)) = k(2k+1)$.  If $n=d_1 d_2$, and $d_1,d_2 \geq 2$ (both even) with $(d_1,d_2) \neq (2,2)$ and $d_1 \leq \sqrt{n}$, then
 \begin{align*}
 \dim(\fsp(V_1) \times \fsp(V_2)) &= \binom{d_1+1}{2} + \binom{d_2+1}{2} < \frac{(d_1+1)^2 + (d_2+1)^2}{2} \\
 &= \frac{d_1{}^2+d_2{}^2}{2} + d_1+d_2+1
 \leq \frac{2^4+n^2}{2^3} + 2+\frac{n}{2}+1 \\
 &= \frac{n^2}{8} + \frac{n}{2} + 5 < \binom{n-2}{2} + 1 = c_0(n), \qquad \forall n \geq 9.
 \end{align*}
 For the remaining $n=8$ case, $\dim(\fsp(V_1) \times \fsp(V_2)) = 13 < c_0(8) = 16$.
 
\end{itemize}
 
Finally, consider $\psi: \ff \to \fgl(\bbV)$ irreducible simple.  If $n = \dim(\bbV) \geq a := \dim(\ff)$, then $c_0(n) \geq \binom{n-2}{2} + 1 \geq \binom{a-2}{2} + 1 > a$ for $a \geq 6$, so $\ff$ is inadmissible.  (Clearly, $\ff = \fsl(2,\bbC)$ is also inadmissible.)  The orthogonal irreps $\psi : \ff \to \fso(\bbV)$ with $\dim(\bbV) < \dim(\ff) < \dim(\fso(\bbV))$ are classified in Appendix \ref{A:irreps}.  This list consists of:
 \begin{itemize}
 \item $(B_3,\lambda_3)$: By triality (i.e. a $D_4$ outer automorphism), $B_3 \to \fso(\bbV_{\lambda_3}) \cong \fso(8,\bbC)$ is equivalent to the standard reducible inclusion $\fso(7,\bbC) \subset \fso(8,\bbC)$.
 \item $(B_4,\lambda_4)$: $\dim(\bbV_{\lambda_4}) = 16$ and $92 = c_0(16) > \dim(B_4) = 36$.
 \item $(C_\rkg,\lambda_2)$: $n = \dim(\bbV_{\lambda_2}) = \binom{2\rkg}{2} -1$, and $c_0(n) > \dim(C_\rkg) = \binom{2\rkg+1}{2}$ for $\rkg \geq 3$.
 \item $(F_4,\lambda_4)$: $n=\dim(\bbV_{\lambda_4}) = 26$ and $c_0(n) = 277 > \dim(\ff) = 52$.
 \item $(G_2,\lambda_1)$:  According to \cite[Theorem 4.3]{CI2008}, branching $\Weyl$ from $\fso(7,\bbC)$ down to $\fg_2$ does not yield a trivial factor, so $\fg_2 \subset \fso(7,\bbC)$ is not visible.  Since $c_0(7) = 11$, then any reductive subalgebra is inadmissible, since these are at most 8-dimensional by Example \ref{E:g2-red-subalg}.
 \end{itemize} 
 
From the above analysis, for $n \geq 5$, any maximal reductive subalgebra of $\fso(n,\bbC)$ with dimension greater than $c_0(n)$ has dimension $\binom{n-1}{2}$, and moreover:

 \begin{theorem} \label{T:so-subalg} Any proper reductive subalgebra $\fs \subset \fso(n,\bbC)$ has dimension at most $\binom{n-1}{2}$ for $n \geq 5$ or $n=3$.   For such $n$, if $\dim(\fs) = \binom{n-1}{2}$, then: (i) when $n \neq 8$, $\fs$ is conjugate to $\fso(n-1,\bbC)$; (ii) when $n=8$, $\fs$ is equivalent to $\fso(7,\bbC)$ via an automorphism of $\fso(8,\bbC)$.
 \end{theorem} 
 
 This recovers a result of Montgomery--Samelson \cite[Lemmas 4 \& 7]{MS1942} and slightly clarifies the $n=8$ case.
The subalgebra $\fso(n-1,\bbC) \subset \fso(n,\bbC)$ does not annihilate a Weyl tensor -- see Appendix \ref{A:branch}.  By Theorem \ref{T:so-subalg}, any proper reductive subalgebra of $\fso(n-1,\bbC)$ has dimension at most $\binom{n-2}{2} < c_0(n)$ for $n \geq 6$, so is inadmissible.  For $n=5$, the maximal reductive subalgebras of $\fso(4,\bbC)$ have dimension $4 = c_0(5)$.  This completes the proof of Theorem \ref{T:main-thm} in the Riemannian case.  Using \eqref{E:Riem-subalg}, most of Theorem \ref{T:subalg-classify} (Riemannian case) has also been proven.  It is clear that $\bbC \times \fso(n-2,\bbC) \subset \fso(n,\bbC)$ admits in $\fso(n)$ only the real form $\fso(2) \times \fso(n-2)$.  It remains to classify real forms of $\fgl(\rkg,\bbC) \subset \fso(2\rkg,\bbC)$ in $\fso(2\rkg)$ (for $\rkg = 2,3,4$) and this is done in Section \ref{S:real}.
 
 Let us exhibit all maximally degenerate Weyl tensors.  Take $\bbR^n$ with standard metric, and standard basis $\{ \omega^i \}$ of $(\bbR^n)^*$.  By Appendix \ref{A:branch}, there is a unique Weyl line which is invariant under $\fso(2) \times \fso(n-2)$.   Explicitly, this line is spanned by:
 \begin{align} \label{E:W-Riem1}
 \phi = \binom{n-2}{2} (\omega^1 \wedge \omega^2)^2  - \frac{n-3}{2} \sum_{i=1}^2 \sum_{a=3}^n (\omega^i \wedge \omega^a)^2 + \sum_{3 \leq a < b \leq n} (\omega^a \wedge \omega^b)^2.
 \end{align}
 
 Let $\rkg \geq 2$.  There is a unique Weyl line invariant under $\fu(\rkg) \subset \fso(2\rkg)$, which is maximally degenerate when $\rkg=2,3,4$.  First, identify $X = A+iB \in \fu(\rkg)$ with $\begin{pmatrix} A & -B\\ B & A \end{pmatrix} \in \fso(2\rkg)$.  On indices, define $\bar{a} := a+\rkg$ for $1 \leq a \leq \rkg$.  Consider the following three $\fu(\rkg)$-invariants in $\bigodot^2 (\bigwedge^2 (\bbR^{2\rkg})^*)$:
 \begin{align*}
 & I_1 = \sum_{1 \leq i < j \leq 2\rkg} (\omega^i \wedge \omega^j)^2, \qquad I_2 = \left( \sum_{a=1}^\rkg\omega^a \wedge \omega^{\bar{a}} \right)^2, \\
 & I_3 = \sum_{1 \leq a < b \leq \rkg} \cA_{2,3,4}\l\{ (\omega^a \wedge \omega^b)(\omega^{\bar{a}} \wedge \omega^{\bar{b}}) \r\},
 \end{align*}
 where $\cA$ denotes skew-symmetrization. The distinguished Weyl line is spanned by
 \begin{align} \label{E:W-Riem2}
 \phi = I_1 - (2\rkg-1) (I_2 + 2I_3).
 \end{align}
 When $n=2\rkg=8$, \eqref{E:W-Riem1} and \eqref{E:W-Riem2} are related via an outer automorphism of $\fso(8)$.
  
 \section{The Lorentzian case}
 \label{S:Lorentzian}
 
 As described in the Introduction, the homogeneous non-conformally flat Lorentzian metric $\metric = \metric_{\rm pp}^{(3,1)} + \metric_{\rm euc}^{(n-4,0)}$ has stabilizer subalgebras of dimension
 \[
 c_0(n) = c(n) - n = \binom{n-1}{2} + 4 -n = \binom{n-2}{2} + 2,
 \]
 and must annihilate a nonzero Weyl tensor.  The Lie algebra structure of the corresponding visible subalgebra of $\fso(1,n-1)$ will be given explicitly below.  It remains to show that $c_0(n) + 1$ is not realizable.  The case $n = 4$ was treated in detail in \cite{KT2013}, so we restrict attention to $n \geq 5$.

 \subsection{General dimensions}
 
 We will use the following theorem of Mostow \cite[Chp.\ 6, Thm.\ 1.9]{OV1994}:
 
 \begin{theorem}[Mostow \cite{Mostow1961}]
 A non-semisimple maximal subalgebra of a real semisimple Lie algebra is either parabolic or coincides with the centralizer of a pseudo-toric subalgebra.
 \end{theorem}
 
 An abelian subalgebra $\fa \subset \fg$ is {\em pseudo-toric} if the subgroup $\exp(\tad\, \fa) \subset \opn{Int}(\fg)$ is compact.  Thus, $\fa$ consists of semisimple elements.  Upon complexification, we have by Example \ref{E:ab-ss} that its centralizer is a regular reductive subalgebra.  We classified such visible subalgebras in our study of the Riemannian case, and when $6 \neq n \geq 5$, these have dimension at most $c_0(n) - 1$, so are inadmissible. The $n=6$ case is exceptional.  In Section \ref{S:real}, we will show that there is no Lorentzian version of the inclusion $\fgl(3,\bbC) \subset \fso(6,\bbC)$ analogous to $\fu(3) \subset \fso(6)$.  The only subalgebra of $\fgl(3,\bbC)$ of dimension $c_0(6) = 8$ is $\fsl(3,\bbC)$, and no Lorentzian analogue of $\fsl(3,\bbC) \subset \fso(6,\bbC)$ exists either. Indeed, it is easy to check that $\fgl(3,\bbC)$ is reconstructed uniquely from $\fsl(3,\bbC)$: it is equal to the sum of $\fsl(3,\bbC)$ and its centralizer in $\fso(6,\bbC)$. Thus, the existence of a Lorentzian analogue of $\fsl(3,\bbC)$ would imply the existence of a Lorentzian analogue of $\fgl(3,\bbC)$ via the same construction. But we know there are no Lorentzian versions of $\fgl(3,\bbC)$.

 \newcommand\pmat[1]{\begin{pmatrix} #1 \end{pmatrix}}
 
  Up to conjugacy, there is a unique parabolic subalgebra $\fp_1 \subset \fso(1,n-1)$ with $\fp_1 = \fr_0 \ltimes \fr_1$, where $\fr_0 \cong \bbR \times \fso(n-2)$ and $\fr_1 \cong \bbR^{n-2}$.  Let $\{ e_i \}_{i=0}^{n-1}$ be a basis of $\bbR^{1,n-1}$ with dual basis $\{ \omega^i \}_{i=0}^{n-1}$ and with respect to which the metric takes the form $\pmat{ 0 & 0 & 1\\ 0 & I_{n-2} & 0\\ 1 & 0 & 0 }$.  Then
 \[
 \fp_1 = \l\{ \pmat{ r & -v^T & 0 \\ 0 & R & v\\ 0 & 0 & -r} : r \in \bbR,\, R \in \fso(n-2),\, v \in \bbR^{n-2} \r\}.
 \]
 Here, $\fr_0$ and $\fr_1$ correspond to the $(r,R)$ and $v$ terms, respectively.  The $\fr_0$-action on $\fr_1$ induced by the Lie bracket can be written symbolically as $(r,R) \cdot v = rv + Rv$.  Writing $e_i^j = e_i \otimes \omega^j$, we have
 \begin{itemize}
 \item $\bbR$ spanned by $e_0^0 - e_{n-1}^{n-1}$;
 \item $\fso(n-2)$ spanned by $e_i^j - e_j^i$ for $1 \leq i < j \leq n-2$;
 \item $\fr_1$ spanned by $e_0 \ot \omega^i - e_i \ot \omega^{n-1}$ for $1 \leq i \leq n-2$.
 \end{itemize}

 Since
 $\dim(\fp_1) = c_0(n) + n-3 \geq c_0(n) + 2$ for $n \geq 5$, 
 then $\fp_1$ is not visible.  Let $\fs \subset \fp_1$ be any visible subalgebra. Consider the natural projection
$\pi \colon \fs \to \fr_0$ induced by the projection $\fp_1\to \fp_1/\fr_1=\fr_0$.
 Two cases arise:
 \Ben
 \item $\fso(n-2) \subset \pi(\fs)$: We have $\ker(\pi) = \fs\cap \fr_1$ and is $\pi(\fs)$-invariant.  Since $\fso(n-2) \subset \pi(\fs)$, then either: (i) $\fs\cap \fr_1 = 0$, so $\dim(\fs) < c_0(n)$; or (ii) $\fs\cap \fr_1 = \fr_1$, so $\fs = \pi(\fs) \ltimes \fr_1$ and $\dim(\fs) \geq c_0(n) + 2$ for $n \geq 6$.  Such $\fs$ are inadmissible.

 \item $\fso(n-2) \not\subset \pi(\fs)$: For $n \geq 7$ or $n=5$, any proper subalgebra of $\fso(n-2)$ (which, by compactness, is reductive) has dimension at most $\dim(\fso(n-3))$.  In these cases, $\dim(\fs) = \dim(\pi(\fs)) + \dim(\fs \cap \fr_1) \leq 1 + \binom{n-3}{2} + n-2 = c_0(n)$.  Equality holds only when $\pi(\fs) \cong \bbR \op \fso(n-3)$ and $\fs \supset \fr_1$. Therefore, $\fs$ is conjugate to:
 \begin{align} \label{E:lor-s}
 \fs(n) = (\bbR \op \fso(n-3)) \ltimes \bbR^{n-2}, \qquad n \geq 7 \qbox{ or } n=5.
 \end{align}
 \Een
  
For $n \geq 4$, we note that $\fs(n)$ stabilizes the degenerate (but not totally null) 2-plane spanned by $\{ e_0, e_1 \}$ which is tangent to the null cone in $\bbR^{1,n-1}$.
  
 \begin{theorem} \label{T:pp-W} For $n \geq 4$, there is a unique Weyl line preserved by  $\fs(n) = (\bbR \op \fso(n-3)) \ltimes \bbR^{n-2} \subset \fso(1,n-1)$.  Letting $\fso(n-3)$ act on $\tspan\{ e_i \}_{i=2}^{n-2}$, we have $\fs(n) = \fco(\phi)$, where
 \begin{align} \label{E:W-Lor}
 \phi = -(n-3)(\omega^1 \wedge \omega^{n-1})^2 + \sum_{i=2}^{n-2} (\omega^i \wedge \omega^{n-1})^2.
 \end{align}
 \end{theorem}
 
 \begin{proof} Let $n \geq 5$.  Since $\fso(n-3)$ is compact and $\fr_1 = \bbR^{n-2}$ is nilpotent in $\fr = \fso(1,n-1)$, then if $\fso(n-3) \ltimes \fr_1$ preserves a nonzero $\phi \in \Weyl$ up to scale, it must annihilate $\phi$.  The space of annihilated elements $\Weyl^{\fr_1}$ is isomorphic to the zeroth cohomology group $H^0(\fr_1,\Weyl)$, whose complexification $H^0((\fr_1)_\bbC,\Weyl_\bbC) = H^0(\fr_1,\Weyl) \ot \bbC$ is easily computed via Kostant's Bott--Borel--Weil theorem \cite{Kos1961, BE1989, CS2009}.
 \[
 \begin{array}{|c|c|c|c|} \hline
 && \multicolumn{2}{c|}{\Weyl^{\fr_1} \ot \bbC}\\ \hline
 n & \mbox{$\fso(n,\bbC)$-weight of $\Weyl_\bbC$} & \mbox{$\fso(n-2,\bbC)$-module} & \mbox{Branching to $\fso(n-3,\bbC)$}\\ \hline\hline
 \geq 7 & 2\lambda_2 & \bigodot^2_0(\bbC^{n-2}) & \bigodot^2_0 (\bbC^{n-3}) \op \bbC^{n-3} \op \bbC\\
 6 & 2\lambda_2 + 2\lambda_3 & \bbC^3 \otimes \bbC^3 & \bigodot^2_0(\bbC^3) \op \bbC^3 \op \bbC\\
 5 & 4\lambda_2 & \bigodot^2_0(\bbC^3) & \bigodot^2_0(\bbC^2) \op \bbC^2 \op \bbC\\ \hline
 \end{array}
 \]
 (For $n=6$, we use $\fso(4,\bbC) \cong \fso(3,\bbC)_L \times \fso(3,\bbC)_R$, and branch $\bbC^3 \otimes \bbC^3$ to $\fso(3,\bbC)$, embedded reducibly in $\fso(4,\bbC)$.)
 Thus, the subspace of $\Weyl$ annihilated by $\fso(n-3) \ltimes \fr_1$ is 1-dimensional.
  For $\phi$ as in \eqref{E:W-Lor}, we verify $\phi \in \Weyl$, $\fso(n-3) \ltimes \fr_1 \subset \fann(\phi)$, and the $\bbR$-factor scales $\phi$.
  
 For $n=4$, we saw in \cite[Sec.~ 5.1.3]{KT2013} that all maximal annihilators in $\fco(1,3)$ are conjugate and 3-dimensional, so this also holds for all maximal visible subalgebras of $\fso(1,3)$.  For $\phi$ as in \eqref{E:W-Lor}, $\fs(4) = \fco(\phi)$ is 3-dimensional, so this proves the result.
  \end{proof}

 Let us also record the following:
 
 \begin{prop} \label{P:so(n-2)}
 For $n \geq 5$, there is a unique Weyl line invariant under $\fso(n-2) \subset \fso(1,n-1)$, spanned by
 \begin{align*}
 \phi &=\binom{n-2}{2} (\omega^0 \wedge \omega^{n-1})^2 +  (n-3)\sum_{i=1}^{n-2} (\omega^0 \wedge \omega^i)(\omega^{n-1} \wedge \omega^i) \\
 &\qquad - \sum_{1 \leq i < j \leq n-2} (\omega^i \wedge \omega^j)^2.
 \end{align*}
 The subalgebra $\fso(n-2) \ltimes \fr_1$ is not contained in a visible subalgebra.
 \end{prop}

 \begin{proof}
 Branching $\Weyl_\bbC$ from $\fso(n,\bbC)$ to $\fso(n-2,\bbC)$ (see Appendix \ref{A:branch}) yields a 1-dimensional trivial factor.  For $\phi$ above, we verify that $\phi \in \Weyl$, and $\fso(n-2) \subset \fann(\phi)$, but $\fr_1$ does not scale $\phi$.
 \end{proof}

 \subsection{Low dimensions}

 For $n=5$, aside from $\fs(5)$, it remains to consider the subalgebra $\fso(3) \ltimes \fr_1 \subset \fp_1$ which has dimension 6.  By Proposition \ref{P:so(n-2)}, this is not a visible subalgebra.
 
For $n=6$, aside from $\fs(6)$, it remains to study subalgebras $\fk \subset \fp_1$ of dimension $\ge c_0(6) = 8$. We can assume that $\fso(4) \not\subset \fk$, as otherwise $\fk$ would necessarily coincide with $\fs(6)$. Consider the natural projection $\pi\colon \fp_1\to\fr_0 = \fso(4)\times\bbR$. Since $\dim(\fr_0) = 7$, we must have $\fk \cap \fr_1 \neq 0$. Since $\fr_1$ is commutative, $\fk\cap\fr_1$ is invariant with respect to the natural action of $\pi(\fk)$ on $\fr_1$. The dimension count immediately implies that $\fk\cap\fr_1=\fr_1$, i.e. $\fk\supset \fr_1=\bbR^4$. Indeed, if $\dim(\fk\cap\fr_1) \leq 3$, then $\pi(\fk)$  lies in the stabilizer of $\fk\cap\fr_1$ and is at most 4-dimensional (since $\fso(3) \subset \fso(4)$ is the largest proper subalgebra acting reducibly).  Then $\dim( \fk ) = \dim( \pi(\fk) ) +\dim(\fk\cap\fr_1)\le 7$.  Thus, $\fk$ is completely determined by $\pi(\fk)\subset \fr_0$, and it remains to study the following subalgebras $\fk\subset \fp_1$:
  \Ben
  \item $(\bbR \op \fso(3)_R) \ltimes \bbR^4$ and $(\fso(2)_L \times \fso(3)_R) \ltimes \bbR^4$ (both dimension 8),
  \item $(\bbR \op (\fso(2)_L \times \fso(3)_R)) \ltimes \bbR^4$ (dimension 9),
  \Een
  and similar subalgebras involving $\fso(3)_L$.
 Here we used the isomorphism $\fso(4) \cong \fso(3)_L \times \fso(3)_R$.

  \begin{prop}
 If $\phi \in \Weyl$ is preserved up to scale by $\fso(3)_R \ltimes \bbR^4 \subset \fso(1,5)$, then $\phi = 0$.  
  \end{prop}

 \begin{proof}  From the proof of Theorem \ref{T:pp-W}, in dimension 6 we have $\Weyl^{\fr_1} \otimes \bbC \cong H^0(\fr_1,\Weyl) \otimes \bbC \cong \bbC^3 \otimes \bbC^3$ as a module for $\fso(4,\bbC) \cong \fso(3,\bbC)_L \times \fso(3,\bbC)_R$ modules. Thus, $\Weyl^{\fr_1} \cong \bbR^3 \ot \bbR^3$ as a module for $\fso(4) \cong \fso(3)_L \times \fso(3)_R$.
 Branching this representation to $\fso(3)_R$ yields three copies of the standard representation $\bbR^3$.  Hence, no trivial factors exist, which proves the claim.
 \end{proof}
 
 A similar statement holds for $\fso(3)_L \ltimes \bbR^4$, so $\fs(6)$ is the maximal visible subalgebra in $\fp_1$.

\section{Real forms of semisimple Lie algebras and their reductive subalgebras}
\label{S:real}

In this section we describe how to classify all real forms of a given reductive subalgebra in a complex simple Lie algebra. Our motivating example is the family of real forms for the subalgebra $\fgl(\rkg,\bbC) \subset \fso(2\rkg,\bbC)$, $\rkg \geq 2$. The definitions and main results are formulated for arbitrary semisimple Lie algebras, but for simplicity all examples are given only for classical Lie algebras. 

\subsection{Real forms, anti-involutions and involutions}
Let $\fg$ be an arbitrary complex Lie algebra with underlying real Lie algebra $\fg_\bbR$. Recall that a real form of $\fg$ is a real subalgebra $\fu \subset \fg_\bbR$ such that $\fg_\bbR = \fu \oplus i\fu$ (direct sum over reals). Each real form $\fu$ defines a unique \emph{anti-involution} $\sigma\colon \fg\to\fg$, where $\sigma(x+iy)=x-iy$, for all $x,y\in\fu$. Conversely, each anti-involution $\sigma$ of $\fg$ defines the real form $\fg_{\sigma}=\{x\in\fg\mid \sigma(x)=x\}$. Thus, describing all real forms of a given complex Lie algebra $\fg$ is equivalent to describing all its anti-involutions $\sigma$.

Suppose now that $\fg$ is complex semisimple. A real form $\fu$ is \emph{compact}, if the Killing form of $\fu$ is negative definite. We call an anti-involution $\tau$ of $\fg$ \emph{compact}, if the corresponding real form $\fg_{\tau}$ is compact. It is well-known \cite{Helgason1978,Knapp2002}, that compact involutions exist for any complex semisimple Lie algebra, and they are all conjugate by the group $\operatorname{Int}(\fg)$ of inner automorphisms of $\fg$.
\begin{center}
\begin{table}[h]
\begin{tabular}{|c|c|c|}
\hline
Lie algebra & Compact anti-involution & Real form \\\hline
\hline
$\fsl(n,\bbC)$ & $X \mapsto -\overline X^{t}$ & $\fsu(n)$ \\
$\fso(n,\bbC)$ & $X\mapsto \overline X$ & $\fso(n)$ \\
$\fsp(2\rkg,\bbC)$ & $X \mapsto -\overline X^{t}$ & $\fsp(\rkg)$ \\
\hline
\end{tabular}
\caption{Compact anti-involutions for classical complex Lie algebras}
\label{F:anti-invol-classical}
\end{table}
\end{center}

Let $\sigma$ be an arbitrary anti-involution of $\fg$.  It is known~\cite[Theorem 6.16]{Knapp2002} that we can always find a compact anti-involution $\tau$ that commutes with $\sigma$, in which case $\theta=\sigma\tau=\tau\sigma$ is an involution. Conversely, if $\theta$ is an arbitrary involution, then we can always find a compact anti-involution $\tau$ commuting with it, and then $\sigma =\theta\tau=\tau\theta$ is clearly an anti-involution. This defines a one-to-one correspondence between the $\operatorname{Aut}(\fg)$-conjugacy classes of anti-involutions and involutions of $\fg$. 

In particular, $\operatorname{Aut}(\fg)$-conjugacy classes of involutions classify real forms of $\fg$. The reason why we want to deal with involutions instead of anti-involutions is obvious: involutions form a much easier class of objects, which can be explicitly described in matrix notation for classical Lie algebras or in terms of root systems in both classical  and exceptional cases.

 Any involution $\theta$ is uniquely determined by its stationary subalgebra $\fg_{\theta}=\{x\in\fg\mid \theta(x)=x\}$. 

Let us explicitly describe all involutions of classical Lie algebras.  For $\fsl(n,\bbC)$, they are:
\begin{itemize}
\item conjugations by matrices $A\in \GL(n,\bbC)$ with $A^2=1$ (which implies that $A$ is diagonalizable with eigenvalues $\pm1$);
\item minus transposition with respect to an arbitrary non-degenerate symmetric or skew-symmetric bilinear form.
\end{itemize}
All involutions of $\fso(n,\bbC)$ for $n\ne 8$ and all involutions of $\fsp(2\rkg,\bbC)$, $\rkg\ge 1$ are conjugations by matrices $A$ from $\operatorname{O}(n,\bbC)$ and $\Sp(2\rkg,\bbC)$ respectively with $A^2=\pm 1$. We exclude here the case of $\fso(8,\bbC)$, as there are other involutions due to the additional symmetry of the root system $D_4$.

Define the following notation:
\[
I_{p,q} = \begin{pmatrix} I_p & 0 \\ 0 & -I_q \end{pmatrix}, \quad
J_k = \begin{pmatrix} 0 & I_k \\ -I_k & 0 \end{pmatrix}, \quad
K_{p,q} = \begin{pmatrix} I_{p,q} & 0 \\ 0 & I_{p,q} \end{pmatrix}.
\]
Also, denote by $\fu^*(\rkg,\bbH)$ the set of skew-Hermitian $\rkg\times \rkg$ matrices over the quaternions $\bbH$ (also denoted by $\fso^*(2\rkg)$ or $\fso(\rkg,\bbH)$ in the literature).
In Table \ref{F:g-real}, we list representatives of $\operatorname{Aut}(\fg)$-conjugacy classes of involutions for classical Lie algebras $\fg$ \cite[Thm.~1.4, Sec.~4.1]{OV1994}. The representatives are chosen in such a way that they commute with the compact anti-involutions given in Table \ref{F:anti-invol-classical}.
\begin{center}
\begin{table}[h]
\begin{tabular}{|l|l|c|c|}
\hline
\multicolumn{1}{|c|}{Lie algebra $\fg$} & \multicolumn{1}{|c|}{Involution $\theta$} & Stabilizer $\fg_{\theta}$ & Real form of $\fg$\\
\hline\hline
$\fsl(n,\bbC)$, $n=p+q$ & $ X \mapsto I_{p,q}XI_{p,q}$ & $\fsl(p,\bbC)+\fsl(q,\bbC)+\bbC$ & $\fsu(p,q)$ \\ 
$\fsl(n,\bbC)$ & $ X \mapsto -X^{t}$ & $\fso(n,\bbC)$ & $\fsl(n,\bbR)$ \\ 
$\fsl(n,\bbC)$, $n=2k$ & $X\mapsto -J_k X^{t} J_k^{-1}$ & $\fsp(n,\bbC)$ & $\fsl(k,\bbH)$ \\
\hline
$\fso(n,\bbC)$, $n=p+q$ & $ X \mapsto I_{p,q}XI_{p,q}$ & $\fso(p,\bbC)+\fso(q,\bbC)$ & $\fso(p,q)$ \\ 
$\fso(n,\bbC)$, $n=2\rkg$ & $X\mapsto J_\rkg X J_\rkg^{-1}$ & $\fgl(\rkg,\bbC)$ & $\fu^*(\rkg,\bbH)$\\
\hline
$\fsp(2\rkg,\bbC)$, $\rkg=p+q$ & $ X \mapsto K_{p,q}XK_{p,q}$ & $\fso(p,\bbC)+\fso(q,\bbC)$ & $\fsp(p,q)$ \\ 
$\fsp(2\rkg,\bbC)$ & $X\mapsto J_\rkg X J_\rkg^{-1}$ & $\fgl(\rkg,\bbC)$ &  $\fsp(2\rkg,\bbR)$ \\
\hline
\end{tabular}
\label{F:g-real}
\caption{Involutions and real forms of classical complex Lie algebras}
\end{table}
\end{center}

\subsection{Real forms of self-normalizing subalgebras}

 A subalgebra $\fk \subset \fg$ is {\em self-normalizing} if it coincides with its own normalizer.  A {\em real form} of $\fk\subset\fg$ refers to an anti-involution of $\fg$ that preserves $\fk$.  The above correspondence between anti-involutions and involutions of semisimple Lie algebras is valid also for pairs of Lie algebras $(\fg,\fk)$, where $\fg$ is semisimple and $\fk$ is its self-normalizing reductive subalgebra \cite{Kom1990}. Namely, if $\sigma$ is an arbitrary real form of $(\fg,\fk)$, then there exists a compact anti-involution $\tau$ of this pair that commutes with $\sigma$ and $\theta=\tau\sigma$ is an involution of $\fg$ preserving $\fk$. Vice versa, if $\theta$ is an arbitrary involution preserving $\fk$, there exists a compact anti-involution that commutes with $\theta$ and preserves $\fk$. Moreover, this defines a one-to-one correspondence between conjugacy classes of involutions and anti-involutions of $(\fg,\fk)$ considered up to $\operatorname{Aut}(\fg,\fk)$, which is the subgroup of $\operatorname{Aut}(\fg)$ which stabilizes $\fk$. 

As an example, let us describe all real forms of the pair $(\fso(2\rkg,\bbC),\fgl(\rkg,\bbC))$. Here, $\fk = \fgl(\rkg,\bbC)$ is defined as the subalgebra of $\fg = \fso(2\rkg,\bbC)$ which stabilizes the decomposition $\bbC^{2\rkg} = V\op V^*$ into a direct sum of a pair of isotropic subspaces.  This notation is due to the fact that one of these subspaces is equivalent to the standard representation of $\fgl(\rkg,\bbC)$, while the other one is equivalent to its dual representation.

In the generic case when $\rkg\ge 5$, there are four types of involutions of $\fso(2\rkg,\bbC)$ preserving $\fgl(\rkg,\bbC)$. They all have the form $X\mapsto AXA^{-1}$, where:
 \begin{table}[h]
 \begin{tabular}{|c|c|c|c|c|}\hline
 & $A$ & $A^2$ & $V,V^*$ & \begin{tabular}{c} Real form of \\ $\fgl(\rkg,\bbC) \subset \fso(2\rkg,\bbC)$ \end{tabular}\\ \hline\hline
 (a) & $\begin{pmatrix} I_{p,q} & 0 \\ 0 & I_{p,q} \end{pmatrix}$, $\rkg=p+q$ & $1$ & both preserved & $\fu(p,q) \subset \fso(2p,2q)$\\
 (b) & $\begin{pmatrix} iI_{p,q} & 0\\ 0 & -iI_{p,q} \end{pmatrix}$, $\rkg=p+q$ & $-1$ & both preserved & $\fu(p,q) \subset \fu^*(\rkg,\bbH)$\\
 (c) & $\begin{pmatrix} 0 & E_\rkg\\ E_\rkg & 0 \end{pmatrix}$ & $1$ & interchanged  & $\fgl(\rkg,\bbR) \subset \fso(\rkg,\rkg)$\\
 (d) & $\begin{pmatrix} 0 & J_k \\ J_k & 0 \end{pmatrix}$, $\rkg=2k$ & $-1$ & interchanged & $\fgl(k,\bbH) \subset \fu^*(2k,\bbH)$ \\
\hline
 \end{tabular}
 \caption{Involutions $X \mapsto AXA^{-1}$ of $\fso(2\rkg,\bbC)$ preserving $\fgl(\rkg,\bbC)$ when $\rkg \geq 5$}
 \end{table}
 
\begin{example}
Suppose that $A$ interchanges $V$ and $V^*$ and suppose that $A^2 = -1$.  Then in a basis adapted to $(V,V^*)$, $A = \begin{pmatrix} 0 & M\\ -M^{-1} & 0 \end{pmatrix}$, where $M^t = - M$.  Let $\varphi \in \operatorname{Aut}(\fg,\fk)$, so $\varphi(X) = BXB^{-1}$.  If $B$ preserves both $V$ and $V^*$, then it is of the form $B = \begin{pmatrix} \alpha & 0\\ 0 & (\alpha^{-1})^{t} \end{pmatrix}$, where $\alpha \in \GL(n,\bbC)$.  Then $\varphi \theta \varphi^{-1}(X)$ is conjugation of $X$ by $\begin{pmatrix} 0 & \alpha M \alpha^t \\ -(\alpha^{-1})^t M^{-1} \alpha^{-1} & 0 \end{pmatrix}$, and we can assume that $M=J_k$.
 \end{example}
  
Although in the case $\rkg=3$ the pair $(\fso(6,\bbC), \fgl(3,\bbC))$ is the same as $(\fsl(4,\bbC), \fgl(3,\bbC))$, the same classification of involutions still holds, and we get the following list of real forms:
\begin{center}
 \begin{table}[h]
\begin{tabular}{|r|c|c|}
\hline
& Real form of $\fgl(3,\bbC) \subset \fso(6,\bbC)$ & Real form of $\fgl(3,\bbC) \subset \fsl(4,\bbC)$ \\ \hline\hline
(a.1) & $\fu(3)\subset \fso(6)$ & $\fu(3)\subset \fsu(4)$ \\
(a.2) & $\fu(1,2)\subset \fso(2,4)$ & $\fu(1,2)\subset \fsu(2,2)$ \\
(c.1) & $\fu(3)\subset \fu^*(3,\bbH)$ & $\fu(3)\subset \fsu(1,3)$ \\
(c.2) & $\fu(1,2)\subset \fu^*(3,\bbH)$ & $\fu(1,2)\subset \fsu(1,3)$ \\
(d) & $\fgl(3,\bbR)\subset \fso(3,3)$ & $\fgl(3,\bbR) \subset \fsl(4,\bbR)$ \\
\hline
\end{tabular}
\caption{Real forms of $\fgl(3,\bbC) \subset \fso(6,\bbC)$}
\end{table}
\end{center}
 Here we use the following isomorphisms between real forms of $\fso(6,\bbC)$ and $\fsl(4,\bbC)$:
 \begin{align*}
 & \fso(6) \cong \fsu(4),\quad 
 \fso(1,5) \cong \fsl(2,\bbH), \quad
 \fso(2,4) \cong \fsu(2,2), \\
 & \fso(3,3) \cong \fsl(4,\bbR), \quad
 \fu^*(3,\bbH) \cong \fsu(1,3).
 \end{align*}

The above classification also stays the same for $\rkg=2$. Note that in this case the pair $(\fso(4,\bbC), \fgl(2,\bbC))$ is the same as $(\fsl(2,\bbC)\times \fsl(2,\bbC), \fsl(2,\bbC)\times \fso(2,\bbC))$. We get the following 6 real forms in this case:
 \begin{center}
 \begin{table}[h]
 \begin{tabular}{|r|c|c|}
 \hline 
 & \begin{tabular}{c}Real form of \\ $\fgl(2,\bbC) \subset \fso(4,\bbC)$ \end{tabular} & \begin{tabular}{c} Real form of \\ $\fsl(2,\bbC) \times \fso(2,\bbC) \subset \fsl(2,\bbC) \times \fsl(2,\bbC)$ \end{tabular}\\ \hline\hline
 (a.1)  & $\fu(2)\subset \fso(4)$ & $\fsu(2)\times \fso(2) \subset \fsu(2)\times\fsu(2)$ \\
 (a.2)  & $\fu(1,1)\subset \fso(2,2)$ & $\fsl(2,\bbR)\times \fso(2) \subset \fsl(2,\bbR)\times\fsl(2,\bbR)$ \\
 (b)     & $\fgl(1,\bbH)\subset \fu^*(2,\bbH)$ & $\fsu(2)\times \fso(1,1) \subset \fsu(2)\times\fsl(2,\bbR)$ \\ 
 (c.1)  & $\fu(2)\subset \fu^*(2,\bbH)$ & $\fsu(2)\times \fso(2) \subset \fsu(2)\times\fsl(2,\bbR)$ \\
 (c.2)  & $\fu(1,1)\subset \fu^*(2,\bbH)$ & $\fsl(2,\bbR)\times \fso(2) \subset \fsl(2,\bbR)\times\fsu(2)$ \\
 (d)    & $\fgl(2,\bbR)\subset \fso(2,2)$ & $\fsl(2,\bbR)\times \fso(1,1) \subset \fsl(2,\bbR)\times\fsl(2,\bbR)$ \\
 \hline
 \end{tabular}
 \caption{Real forms of $\fgl(2,\bbC) \subset \fso(4,\bbC)$}
 \end{table}
 \end{center}

Finally, when $\rkg=4$, the subalgebra $\fgl(4,\bbC) \subset \fso(8,\bbC)$ is conjugate to $\fso(6,\bbC) \times \fso(2,\bbC)$ by an outer automorphism of $D_4$.  More explicitly, given an isotropic decomposition $\bbC^8 = V\oplus V^*$, consider the (8-dimensional) positive spin representation $\rho : \fso(8,\bbC) \to \fgl(\bbS_+)$ (see~\cite{Harvey1990}).  As vector spaces,
 \begin{equation}\label{spinRep}
 \bbS_+=\bigwedge\nolimits^0 V \op\bigwedge\nolimits^2 V \op\bigwedge\nolimits^4 V.
 \end{equation} 
 The wedge product defines an invariant symmetric bilinear form on $\bbS_+$, so $\rho$ has image in $\fso(\bbS_+) \cong \fso(8,\bbC)$. The subalgebra $\fgl(4,\bbC)\subset \fso(8,\bbC)$ can be defined as a set of elements in $\fso(8,\bbC)$ preserving the above isotropic decomposition. In particular, $\rho|_{\fgl(4,\bbC)}$ decomposes $\bbS_+$ into three invariant subspaces of dimensions 1, 6, and 1. Thus, $\rho|_{\fgl(4,\bbC)}$ preserves the decomposition~\eqref{spinRep} and coincides with $\fso(6,\bbC) \times \fso(2,\bbC)$ \cite[Chp.\ 14]{Harvey1990}. 

We can enumerate all real forms of $\fgl(4,\bbC)\subset \fso(8,\bbC)$ as follows. First, we can construct different real forms using the above list of involutions for generic $\rkg$. Next, we can also construct different real forms of  $\fso(6,\bbC) \times \fso(2,\bbC)$. Then we have to find out when these forms are inequivalent to each other. In fact, it turns out~\cite{Kom1994} that different pairs we obtain in this manner are equivalent if and only if the Lie algebras and the subalgebras in these pairs are isomorphic as abstract Lie algebras. As shown in \cite{Kom1994}, we get the following 12 different real forms:
\begin{table}[h]
 \[
 \begin{array}{|c|c|} \hline
 \mbox{Real forms of $\fso(8,\bbC)$} & \mbox{Real forms of $\fgl(4,\bbC)$ in $\fso(8,\bbC)$}\\ \hline\hline
 \fso(8) & \fu(4) \cong \fso(6) \times \fso(2)\\ \hline
 \fso(1,7) &  \fso(6)\times\fso(1,1) \\
 	& \fso(1,5)\times\fso(2) \\ \hline
 \fso(2,6) \cong \fu^*(4,\bbH) & \fu(1,3) \cong \fu^*(3,\bbH)\\
 	& \fgl(2,\bbH) \cong \fso(1,5) \times \fso(1,1)\\
 	& \fu(4) \cong \fso(6) \times \fso(2) \\
 	& \fu(2,2) \cong \fso(2,4) \times\fso(2) \\ \hline
 \fso(3,5) & \fso(2,4)\times\fso(1,1)\\
	& \fso(1,5)\times\fso(2)\\
     & \fso(3,3)\times\fso(2)\\ \hline
 \fso(4,4) & \fu(2,2) \cong \fso(2,4) \times \fso(2) \\
  & \fgl(4,\bbR) \cong \fso(3,3)\times\fso(1,1)\\ \hline
 \end{array}
 \]
 \caption{Real forms of $\fgl(4,\bbC) \subset \fso(8,\bbC)$}
 \end{table}

Consequently, we immediately obtain the following:
\begin{prop}
 For $\rkg \geq 2$, $\fgl(\rkg,\bbC) \subset \fso(2\rkg,\bbC)$ admits (Lorentzian) real forms $\fk \subset \fso(1,2\rkg-1)$ only when $\rkg = 4$.
\end{prop}

\appendix

\section{Dimensions of irreducible representations}
\label{A:irreps}

 The data in Table \ref{F:irreps} is derived from \cite[Table 5]{OV1990}.
 
\renewcommand\arraystretch{1.2}
 \begin{table}[h]
 $\begin{array}{|l|l|c|ll|} \hline
 \multicolumn{1}{|c|}{\ff} & \mbox{Range} & \dim(\ff) & \multicolumn{2}{c|}{\mbox{Dimensions of fundamental irreps}}\\ \hline\hline
 A_\rkg = \fsl(\rkg+1,\bbC) & \rkg \geq 1 & \rkg(\rkg+2) & \dim(\bbV_{\lambda_k}) = \binom{\rkg+1}{k}, & 1 \leq k \leq \rkg\\
 B_\rkg = \fso(2\rkg+1,\bbC) & \rkg \geq 2 & \binom{2\rkg+1}{2} & \dim(\bbV_{\lambda_k}) = \binom{2\rkg+1}{k}, & 1 \leq k \leq \rkg-1\\
 &&& \dim(\bbV_{\lambda_\rkg}) = 2^\rkg &\\
 C_\rkg = \fsp(2\rkg,\bbC) & \rkg \geq 3 & \binom{2\rkg+1}{2} & \dim(\bbV_{\lambda_k}) = \frac{2\rkg-2k+2}{2\rkg-k+2} \binom{2\rkg+1}{k}, & 1 \leq k \leq \rkg\\
 D_\rkg = \fso(2\rkg,\bbC) & \rkg \geq 4 & \binom{2\rkg}{2} & \dim(\bbV_{\lambda_k}) = \binom{2\rkg}{k}, & 1 \leq k \leq \rkg-2\\
  &&& \dim(\bbV_{\lambda_k}) = 2^{\rkg-1}, & k = \rkg-1, \rkg\\ \hline
 \end{array}$
 \caption{Dimensions of fundamental irreps of classical complex simple Lie algebras}
 \label{F:irreps}
 \end{table}

 \begin{prop} \label{P:orth-irrep}
 The only irreps $\psi : \ff \to \fgl(\bbV_\lambda)$ of a classical simple Lie algebra $\ff$ of type $B,C,D$ having $\dim(\bbV_\lambda) < \dim(\ff)$ are those with the following highest weights $\lambda$:
 \[
 \begin{array}{|c|c|c|c|} \hline
 \ff & \mbox{Range} & \mbox{General} & \mbox{Low dimensional exceptions}\\ \hline\hline
 B_\rkg & \rkg \geq 2 & \lambda_1 & \lambda_\rkg \quad (2 \leq \rkg \leq 6)\\
 C_\rkg & \rkg \geq 3 & \lambda_1, \lambda_2 & \lambda_3 \quad (\rkg = 3)\\
 D_\rkg & \rkg \geq 4 & \lambda_1 & \lambda_{\rkg-1}, \lambda_\rkg \quad (4 \leq \rkg \leq 7)\\ \hline
 \end{array}
 \]
 The self-dual irreps $\psi : A_\rkg \to \fgl(\bbV_\lambda)$ with $\dim(\bbV_\lambda) < \dim(A_\rkg)$ have highest weight $\lambda_r$ with $r = \frac{\rkg+1}{2}$ when $\rkg = 1, 3$ or $5$.
 
 Among {\em all} simple Lie algebras $\ff$, the orthogonal irreps $\psi : \ff \to \fso(\bbV_\lambda)$ satisfying $\dim(\bbV_\lambda) < \dim(\ff) < \dim(\fso(\bbV_\lambda))$ are:
 \[
 (B_\rkg,\lambda_\rkg) \mbox{ for } \rkg = 3,4; \qquad (C_\rkg, \lambda_2) \mbox{ for } \rkg \geq 3; \qquad
 (G_2, \lambda_1); \qquad (F_4,\lambda_4).
 \]
 \end{prop}
 
 \begin{proof}  Given two dominant integral weights $\lambda, \mu$ with $\lambda \geq \mu$, we have $\dim(\bbV_\lambda) \geq \dim(\bbV_\mu)$ by the Weyl dimension formula.  Consequently, it suffices to: (i) identify the fundamental weights $\lambda_k$ satisfying $\dim(\bbV_{\lambda_k}) < \dim(\ff)$ using Table \ref{F:irreps}; and (ii) examine those $\lambda$ which are integral linear combinations of the fundamental weights in (i).  All possibilities below give dimensions $\geq \dim(\ff)$:
 \begin{itemize}
 \item $B_\rkg$: $\dim(\bbV_{2\lambda_1}) = \dim(\bigodot^2_0(\bbV_{\lambda_1})) = \binom{2\rkg+2}{2} - 1$, and $\dim(\bbV_{2\lambda_\rkg}) = \dim(\bigwedge^\rkg V_{\lambda_1}) = \binom{2\rkg+1}{\rkg}$.  Since $\bbV_{\lambda_1} \otimes \bbV_{\lambda_{\rkg}} \cong \bbV_{\lambda_1 + \lambda_{\rkg}} \oplus \bbV_{\lambda_\rkg}$, then $\dim(\bbV_{\lambda_1 + \lambda_\rkg}) = 2^{\rkg+1} \rkg$.
 \item $C_\rkg$: Let $r = 2\rkg$.  Then $\dim(\bbV_{2\lambda_1}) = \dim(C_\rkg) = \binom{r+1}{2}$.  From \cite{OV1990}, 
 \[
 \dim(\bbV_{2\lambda_2}) = \frac{r + 3}{3(r-1)} \binom{r}{2} \binom{r-1}{2}, \qquad \dim(\bbV_{\lambda_1 + \lambda_2}) = \frac{r(r-2)(r+2)}{3}.
 \]
 When $\rkg = 3$, $\dim(\bbV_{2\lambda_3}) = 84$, $\dim(\bbV_{\lambda_1 + \lambda_3}) = 70$, and $\dim(\bbV_{\lambda_2 + \lambda_3}) = 126$.
 \item $D_\rkg$: $\dim(\bbV_{2\lambda_1}) = \dim(\bigodot^2_0(\bbV_{\lambda_1})) = \binom{2\rkg+1}{2} - 1$, and $\dim( \bbV_{\lambda_{\rkg-1} + \lambda_\rkg} ) = \dim( \bigwedge^{\rkg-1} \bbV_{\lambda_1}) = \binom{2\rkg}{\rkg-1}$.
 From \cite{OV1990}, $\dim(\bbV_{2\lambda_{\rkg-1}}) = \dim(\bbV_{2\lambda_\rkg}) = \binom{2\rkg-1}{\rkg-1}$.  Since $\bbV_{\lambda_1} \otimes \bbV_{\lambda_\rkg} \cong \bbV_{\lambda_1 + \lambda_\rkg} \oplus \bbV_{\lambda_{\rkg-1}}$, then $\dim(\bbV_{\lambda_1 + \lambda_\rkg} ) = (2\rkg-1) 2^{\rkg-1}$.  Similarly, $\dim(\bbV_{\lambda_1 + \lambda_{\rkg-1}} ) = (2\rkg-1) 2^{\rkg-1}$.
 \end{itemize}
 
 For $\rkg \geq 2$, an irrep $\psi : A_\rkg \to \fgl(\bbV_\lambda)$ is self-dual iff $\lambda$ is invariant under the duality involution (rotate the Dynkin diagram by 180 degrees).  Equivalently, $\lambda$ is an integral linear combination of 
 \[
 \lambda_1 + \lambda_\rkg, \quad
 \lambda_2 + \lambda_{\rkg-1}, \quad ..., \quad
 \left\{ 
 \begin{array}{lll} \lambda_{r} + \lambda_{r+1}, & r= \frac{\rkg}{2}, & \rkg \mbox{ even};\\
 \lambda_r, & r= \frac{\rkg+1}{2}, & \rkg \mbox{ odd}.
 \end{array} \right.
 \]
 From \cite[Table 5]{OV1990}, we have for $k \leq \left\lfloor \frac{\rkg+1}{2} \right\rfloor$, 
 \[
 \dim(\bbV_{\lambda_{\rkg+1 - k} + \lambda_k}) = \frac{\rkg+2-2k}{\rkg+2-k}\binom{\rkg+1}{\rkg+1-k} \binom{\rkg+2}{k} = \frac{\rkg+2-2k}{\rkg+2} \binom{\rkg+2}{k}^2.
 \]
 We deduce that $\dim(\bbV_{\lambda_{\rkg+1 - k} + \lambda_k}) < \dim(A_\rkg)$ is never true.  On the other hand, for $\rkg$ odd and $r = \frac{\rkg+1}{2}$, we have $\dim(\bbV_r) = \binom{\rkg+1}{r}$.  This is less than $\dim(A_\rkg)$ only when $\rkg = 3$ or $5$.  For $\rkg = 1$, the standard representation is the unique irrep with dimension less than $\dim(A_1)$ (and self-dual).
 
 All irreps in type $B,C$ are self-dual, as is $(D_\rkg, \lambda_1)$.  The duality involution for $D_\rkg$ is non-trivial when $\rkg$ is odd, so among $(D_\rkg,\lambda_{\rkg-1})$, $(D_\rkg,\lambda_\rkg)$ for $4 \leq \rkg \leq 7$, the self-dual ones occur when $\rkg = 4$ or $6$.
 
 Using \cite[Corollary on page 98]{OV1994}, we can immediately identify which of these self-dual irreps are orthogonal and satisfy the given dimension inequality:
 \begin{itemize}
 \item $\ff \ra \fso(\bbV_\lambda)$ is an isomorphism for $(A_3, \lambda_2)$, $(B_\rkg, \lambda_1)$, $(D_\rkg, \lambda_1)$, $(D_4, \lambda_3)$, $(D_4, \lambda_4)$.
 \item Symplectic irreps: $(A_5,\lambda_3)$, $(B_\rkg, \lambda_\rkg)$ for $\rkg=2,5,6$, $(C_\rkg,\lambda_1)$, $(C_3,\lambda_3)$, $(D_6,\lambda_5)$, $(D_6,\lambda_6)$
 \item Not self-dual: $(D_\rkg,\lambda_{\rkg-1})$, $(D_\rkg,\lambda_\rkg)$ for $\rkg=5,7$.
 \end{itemize}
 Also, for $(A_1,\lambda_1)$, we have $2 = \dim(\bbV_{\lambda_1}) = \dim(\fso(\bbV_{\lambda_1}))$ and $\dim(A_1) = 3$.  The remaining irreps are orthogonal and satisfy the given inequality.
 
 The only irreps of exceptional simple Lie algebras $\ff$ satisfying $\dim(\ff) < \dim(\fso(\bbV_\lambda))$ are $(G_2,\lambda_1)$, $(F_4,\lambda_4)$, $(E_6, \lambda_1) \cong (E_6, \lambda_6)$, and $(E_7,\lambda_7)$.  While the first two are orthogonal and satisfy the inequality, $(E_6,\lambda_1)$ is not self-dual, and $(E_7,\lambda_7)$ is symplectic.
 \end{proof}

\section{Branching rules}
\label{A:branch}

 Letting a ``0'' subscript below indicate the totally trace-free part, we recall that
 \begin{align}
 \bigodot\nolimits^2( \bigwedge\nolimits^2 \bbC^n) &\cong \yng(2,2)(\bbC^n) \op \bigwedge\nolimits^4(\bbC^n) \label{E:branch1}\\
 &= \underbrace{\yngtf{2,2}(\bbC^n)}_{\Weyl_\bbC} \op \bigodot\nolimits^2(\bbC^n) \op \bigwedge\nolimits^4(\bbC^n), \nonumber\\
 \bigwedge\nolimits^2 (\bbC^n) \ot \bbC^n &\cong \bigwedge\nolimits^3 (\bbC^n) \op \yng(2,1)(\bbC^n) = \bigwedge\nolimits^3 (\bbC^n) \op \yngtf{2,1}(\bbC^n) \op \bbC^n. \label{E:branch2}
 \end{align}
 We will branch $\Weyl_\bbC$ from $\fso(n,\bbC)$ to several choices of subalgebras $\fk$.  
 \begin{enumerate}
 \item[(a)] \framebox{$\fk = \fso(n-1,\bbC)$}: Writing $\bbC^n = \bbC^{n-1} \op \bbC$ and then using \eqref{E:branch1}, \eqref{E:branch2},  
 \begin{align}
 \bigodot\nolimits^2 (\bigwedge\nolimits^2 \bbC^n) &= \bigodot\nolimits^2( \bigwedge\nolimits^2 \bbC^{n-1} ) \op \bigwedge\nolimits^2 \bbC^{n-1} \ot \bbC^{n-1}\op \bigodot\nolimits^2 (\bbC^{n-1} );\\
 \Weyl_\bbC &= \yngtf{2,2} (\bbC^{n-1}) \op \yngtf{2,1} (\bbC^{n-1}) \op \yngtf{2} (\bbC^{n-1}), \quad n \geq 4.\label{E:Weyl-branch}
 \end{align}
  For $n \geq 6$, each factor above is irreducible for $\fso(n-1,\bbC)$.  For $n=5$, $\yngtf{2,2}(\bbC^4)$ splits into self-dual and anti-self-dual irreducible components (each of dimension 5), while
  \begin{align} \label{E:21-decompose}
  \yngtf{2,1} (\bbC^4) = \l( \bigodot\nolimits^3 \bbS_+ \boxtimes \bbS_-\r) \op \l( \bbS_+ \boxtimes \bigodot\nolimits^3 \bbS_-\r),
  \end{align}
   where $\bbC^4 = \bbS_+ \boxtimes \bbS_-$ in terms of the half-spin representations $\bbS_+ \cong \bbS_- \cong \bbC^2$ for $\fso(4,\bbC)$.  For $n=4$, $\yngtf{2,2} (\bbC^3) \cong \yngtf{2} (\bbC^3)$, and $\yngtf{2,1} (\bbC^3)  = 0$.
 Thus, no trivial factors arise.
 \item[(b)] \framebox{$\fk = \fso(n-2,\bbC)$}: Let $n \geq 5$.  By part (a), $\yngtf{2,2} (\bbC^{n-1})$ produces no trivial factor.
 \begin{itemize}
 \item[(i)] $\yngtf{2} (\bbC^{n-1}) = \bbC \op \bbC^{n-2} \op \yngtf{2} (\bbC^{n-2})$;
 \item[(ii)] $\yngtf{2,1} (\bbC^{n-1}) = \bbC^{n-2} \op \bigwedge^2 \bbC^{n-2} \op \yngtf{2,1} (\bbC^{n-2}) \op \yngtf{2} (\bbC^{n-2})$.
 \end{itemize}
 If $n \geq 7$, all factors in (i) and (ii) are irreducible, as is $\yngtf{2} (\bbC^{n-2})$ for $n=5,6$.  For $n=6$, we have \eqref{E:21-decompose} and $\bigwedge^2 \bbC^4 = \bigwedge^2_+ \bbC^4 \op \bigwedge^2_- \bbC^4$. 
 If $n=5$, we have $\bigwedge^2 \bbC^3 \cong \bbC^3$, and letting $\bbS \cong \bbC^2$ denote the spin representation of $\fso(3,\bbC)$, we have $\yngtf{2,1} (\bbC^3) = \bbC^3 \op \yngtf{2} (\bbC^3) \cong \bigodot\nolimits^2 \bbS \op \bigodot\nolimits^4 \bbS$.
 Thus, for $n \geq 5$, a single trivial factor arises.
 \item[(c)] \framebox{$\fk = \fgl(\rkg,\bbC) \subset \fso(2\rkg,\bbC)$}: Given an isotropic decomposition $\bbC^{2\rkg} = V \op V^*$ (with $V \cong \bbC^\rkg$),
 \begin{itemize}
 \item $\rkg=2$: We have $\Weyl_\bbC = \Weyl_\bbC^+ \op \Weyl_\bbC^- \cong \left( \bigodot^4 \bbS_+ \boxtimes \bbC \right) \op \left( \bbC \boxtimes \bigodot^4 \bbS_- \right)$ as a representation of $\fso(4,\bbC) \cong \fso(3,\bbC)_L \times \fso(3,\bbC)_R$.  While $\Weyl_\bbC^+$ remains irreducible for $\fgl(2,\bbC)$, $\Weyl_\bbC^-$ decomposes into five 1-dimensional weight spaces (of weights $-4,-2,0,2,4$).
 \item $\rkg\geq 3$: We have
 \begin{align*}
 \bigodot\nolimits^2\l( \bigwedge\nolimits^2 (\bbC^{2\rkg})\r) &\cong 
  \bigodot\nolimits^2\l( \bigwedge\nolimits^2 V \op (V \ot V^*) \op \bigwedge\nolimits^2 V^*\r)\\
 &\cong \bigodot\nolimits^2( \bigwedge\nolimits^2 V ) \op \bigodot\nolimits^2 (V \ot V^*) \op \bigodot\nolimits^2 ( \bigwedge\nolimits^2 V^* )\\
 &\qquad \op \l( \bigwedge\nolimits^2 V \ot V \ot V^* \r) \op \l( V \ot V^* \ot \bigwedge\nolimits^2 V^* \r) \\
 &\qquad \op \l( \bigwedge\nolimits^2 V \ot \bigwedge\nolimits^2 V^* \r).
 \end{align*}
 Since $\bigwedge\nolimits^4 (\bbC^{2\rkg}) \cong \bop_{i=0}^4 (\bigwedge\nolimits^i V \ot \bigwedge\nolimits^{4-i} V^*)$, then using \eqref{E:branch1} and \eqref{E:branch2},
 \begin{align*}
 \yng(2,2)(\bbC^{2\rkg}) &\cong \yng(2,2)(V) \op \yng(2,1)(V) \ot V^* \op \bigodot\nolimits^2 (V\ot V^*) \\
 &\qquad \op \yng(2,1)(V^*) \ot V \op \yng(2,2)(V^*) .
 \end{align*}
 Also,
 \begin{align*}
 \bigodot\nolimits^2 (V \ot V^*) &\cong \bigodot\nolimits^2 (\fsl(V) \op \bbC) \cong \bigodot\nolimits^2 (\fsl(V)) \op \fsl(V) \op \bbC \\
 &\cong \l( \bigodot\nolimits^2 V \ot \bigodot\nolimits^2 V^* \r)_0 \op \l( \bigwedge\nolimits^2 V \ot \bigwedge\nolimits^2 V^* \r)_0 \\
 &\qquad \op \fsl(V) \op \bbC \op \fsl(V) \op \bbC,\\
 \yng(2,1)( V ) \ot V^* &\cong \l(  \yng(2,1)( V ) \ot V^* \r)_0 \op \bigodot\nolimits^2 V \op \bigwedge\nolimits^2 V.
 \end{align*}
 By \eqref{E:branch1}, $\Weyl_\bbC\cong \yngtf{2,2}(\bbC^{2\rkg})$ decomposes into $\fgl(\rkg,\bbC)$-irreps as
 \begin{align*}
 \Weyl_\bbC  &\cong \yng(2,2)(V) \op \yng(2,2)(V^*) \op \l(  \yng(2,1)(V) \ot V^* \r)_0 \op \l(  \yng(2,1)(V^*) \ot V \r)_0 \\
 &\qquad
 \op \l( \bigodot\nolimits^2 V \ot \bigodot\nolimits^2 V^* \r)_0 \op \l( \bigwedge\nolimits^2 V \ot \bigwedge\nolimits^2 V^* \r)_0 \\
 &\qquad \op \fsl(V) \op \bigwedge\nolimits^2 V \op \bigwedge\nolimits^2 V^* \op \bbC.
 \end{align*}
 \end{itemize}
 \end{enumerate}
 


 \bibliographystyle{amsplain}

\end{document}